\documentclass{article} %

\usepackage{amsmath}
\usepackage{amsthm}
\usepackage{amssymb}
\usepackage{graphicx}
\usepackage[usenames]{color}
\usepackage{mathabx}
\usepackage{hyperref}

\newtheorem{theorem}{Theorem}[section]
\newtheorem{proposition}[theorem]{Proposition}

\newtheorem{lemma}[theorem]{Lemma}
\newtheorem{definition}[theorem]{Definition}
\newtheorem{remark}[theorem]{Remark}
\numberwithin{equation}{section}

\newcommand{\bydef}{\stackrel{\mbox{\tiny\textnormal{\raisebox{0ex}[0ex][0ex]{def}}}}{=}\,} 
\newcommand{\Z}{\mathbb{Z}}
\newcommand{\N}{\mathbb{N}}
\newcommand{\C}{\mathbb{C}}
\newcommand{\R}{\mathbb{R}}
\newcommand{\imag}{\text{\textup{i}}\:\!}
\newcommand{\bs}[1]{\boldsymbol{#1}}
\newcommand{\ampl}{\text{\textup{ampl}}}
\newcommand{\hx}{\hat{x}}
\newcommand{\hl}{\hat{\lambda}}
\newcommand{\hv}{\hat{v}}
\newcommand{\hr}{\hat{r}}
\newcommand{\G}{\mathcal{G}}
\newcommand{\s}{\textup{\texttt{s}}}
\newcommand{\phase}{{\leftmoon}}
\newcommand{\cont}{{\Sun}}
\newcommand{\ameq}{{\Earth}}
\newcommand{\sstar}{s_\star}
\newcommand{\mustar}{\mu_\star}

\newcommand{\sol}[1]{\mathring{#1}}
\newcommand{\pp}{\phi}
\newcommand{\yy}{\psi}
\newcommand{\stat}[1]{\breve{#1}}
\newcommand{\XS}{\mathcal{S}}

\newcommand{\tH}{\widetilde{H}}
\newcommand{\hH}{\widehat{H}}
\newcommand{\Q}{\mathcal{Q}}
\newcommand{\QQ}{\mathcal{R}}
\newcommand{\PP}{\mathcal{P}}
\newcommand{\A}{A}
\DeclareMathOperator{\range}{Range}
\newcommand{\K}{\text{\textup{K}}}
\newcommand{\y}{y}
\newcommand{\sy}{\stat{\y}}
\newcommand{\Y}{\widetilde{X}}
\newcommand{\LL}{\mathcal{L}}

\newcommand{\checkwstar}{\check{w}_{\star}}
\newcommand{\element}{x}

\def\corcommstyle{\bf\small\tt}

\def\corrl #1<<#2||#3>>{
\if\visiblecomments y
  \begin{quote} {\corcommstyle $<<$COMMENT$>>$ {\color{red}#1\marginpar{!!}}\\$<<$OLD$<<$} \end{quote}

{\color{red} 
 #2
 }

  \begin{quote} {\corcommstyle ==NEW== } \end{quote}
   \noindent\hrulefill
 
\vspace{-10pt} 
 
 \noindent\hrulefill
 
 \vspace{-10pt} 
 
 \noindent\dotfill
 
  #3
  
   \noindent\dotfill 

\vspace{-10pt} 
 
 \noindent\hrulefill
 
 \vspace{-10pt} 
 
 \noindent\hrulefill
  \begin{quote} {\corcommstyle $>>$END$>>$ } \end{quote}
 \else
  #3
 \fi
}

\long\def\longcorrl #1<<#2||#3>>{
\if\visiblecomments y
  \begin{quote} {\corcommstyle $<<$COMMENT$>>$ {\color{red}#1\marginpar{!!}}\\$<<$OLD$<<$} \end{quote}
 
 {\color{red}

  #2
  
  }
  
  \begin{quote} {\corcommstyle ==NEW== } \end{quote}
  
    \noindent\hrulefill
 
\vspace{-10pt} 
 
 \noindent\hrulefill
 
 \vspace{-10pt} 
 
 \noindent\dotfill
 
  #3
  
   \noindent\dotfill 

\vspace{-10pt} 
 
 \noindent\hrulefill
 
 \vspace{-10pt} 
 
 \noindent\hrulefill
  \begin{quote} {\corcommstyle $>>$END$>>$ } \end{quote}
 \else
  #3
 \fi
}

\def\mlabel #1
{
  \if\visiblecomments y
     \marginpar[\flushright \bf \footnotesize #1]{\bf \footnotesize #1}
  \fi
}

\def\flabel #1
{
  \if\visiblecomments y
       \hbox{\bf\footnotesize #1}
  \fi
}

\def\corrq #1<<#2>>{
\if\visiblecomments y
  \begin{quote} {\corcommstyle $<<$COMMENT$>>$ #1\marginpar{!!}\\$<<$BEG$<<$} \end{quote}
  \noindent\hrulefill
 
\vspace{-10pt} 
 
 \noindent\hrulefill
 
 \vspace{-10pt} 
 
 \noindent\dotfill
 
 {\color{red}
  
  #2
  
  }
   
  \noindent\dotfill 

\vspace{-10pt} 
 
 \noindent\hrulefill
 
 \vspace{-10pt} 
 
 \noindent\hrulefill 
  \begin{quote} {\corcommstyle $>>$END$>>$ } \end{quote}
 \else
  #2
 \fi
}

\long\def\longcorrq #1<<#2>>{
\if\visiblecomments y
  \begin{quote} {\corcommstyle $<<$COMMENT$>>$ #1\marginpar{!!}\\$<<$BEG$<<$} \end{quote}
  \noindent\hrulefill
 
\vspace{-10pt} 
 
 \noindent\hrulefill
 
 \vspace{-10pt} 
 
 \noindent\dotfill

  #2

  \noindent\dotfill 

\vspace{-10pt} 
 
 \noindent\hrulefill
 
 \vspace{-10pt} 
 
 \noindent\hrulefill 
  \begin{quote} {\corcommstyle $>>$END$>>$ } \end{quote}
 \else
  #2
 \fi
}

\def\corrc #1<<>>{
\if\visiblecomments y
  \begin{quote} {\corcommstyle $<<$COMMENT$>>$ \color{red} #1\marginpar{!!}} \end{quote}
\fi
}

\def\corre #1<<#2||#3>>{
\if\visiblecomments y
  #3\marginpar{\corcommstyle #1}
 \else
  #3
 \fi
}

\long\def\longcorre #1<<#2||#3>>{
\if\visiblecomments y
  #3\marginpar{\corcommstyle #1}
 \else
  #3
 \fi
}

\def\corrs #1<<#2||#3>>{
\if\visiblecomments y
  #3\marginpar{\corcommstyle #2 $\rightarrow$ #3\\ #1}
 \else
  #3
 \fi
}

\def\corro #1<<#2||#3>>{
#2}

\def\corrn #1<<#2||#3>>{
#3}

\long\def\longcorro #1<<#2||#3>>{
#2}

\long\def\longcorrn #1<<#2||#3>>{
#3}

\long\def\underconstruction #1<<<#2>>>{
\if\visiblecomments y
  \begin{quote} {\corcommstyle $<<$UNDER CONSTRUCTION - BEGIN$>>$ #1\marginpar{!!}} \end{quote}
  #2
  \begin{quote} {\corcommstyle $>>$UNDER CONSTRUCTION - END$>>$ } \end{quote}
 \else
 \fi
}

\def\showcomments{
  \let\visiblecomments y
}

\def\hidecomments{
  \let\visiblecomments n
}

\showcomments
\begin{document}
\title{Rigorous verification of Hopf bifurcations\\ via desingularization and continuation}

\author{
Jan Bouwe van den Berg 
\thanks{
Department of Mathematics, 
VU Amsterdam, 
1081 HV Amsterdam, 
The Netherlands, {\tt janbouwe@few.vu.nl};
partially supported by NWO-VICI grant 639033109.
}
\and
Jean-Philippe Lessard 
\thanks{
Department of Mathematics and Statistics, 
McGill University, 
805 Sherbrooke St W, 
Montreal, QC, H3A 0B9, 
Canada, {\tt jp.lessard@mcgill.ca}; supported by NSERC. 
}
\and 
Elena Queirolo
\thanks{
Department of Mathematics,
Rutgers University,
110 Frelinghuysen Road,
Piscataway, NJ 08854-8019, 
USA.
{\tt elena.Queirolo@rutgers.edu}.
}
}

\date{\today}

\maketitle

\begin{abstract}
In this paper we present a general approach to rigorously validate Hopf bifurcations as well as saddle-node bifurcations of periodic orbits in systems of ODEs. By a combination of analytic estimates and computer-assisted calculations, we follow solution curves of cycles through folds, checking along the way that a single nondegenerate saddle-node bifurcation occurs. Similarly, we rigorously continue solution curves of cycles starting from their onset at a Hopf bifurcation. We use a blowup analysis to regularize the continuation problem near the Hopf bifurcation point. This extends the applicability of validated continuation methods to the mathematically rigorous computational study of bifurcation problems.  
\end{abstract}

\begin{center}
{\bf \small Keywords} \\ \vspace{.05cm}
{ \small Hopf bifurcation $\cdot$ continuation $\cdot$ desingularization $\cdot$ computer-assisted proofs }
\end{center}

\begin{center}
{\bf \small Mathematics Subject Classification (2010)}  \\ \vspace{.05cm}
{\small 37G15 $\cdot$ 65P30 $\cdot$ 65G40 $\cdot$ 34C25 $\cdot$ 37C27 } 
\end{center}

\section{Introduction}
\label{s:intro}
%!TEX root = bifcont.tex

% \corrc I proposed to change $M' \to 1$, $N' \to n$ and $M''$ dissapeared. JP <<>>
% {\color{green}Elena's comments}

% \corrc \bf I think $\omega$ is in fact $\frac{L}{2\pi}$, i.e., the "normalized" period, and not the frequency, right? This mistake seems to have crept into the continuation paper ....arggh.... The name $\omega$ is then perhaps not very suitable? Perhaps replace it by $\tau$? JB <<>>
%
% \corrc I made this change; there are still a couple of places where I need to look at the formulation. Also, the time-axis in most plots needs to be repaired, since $\omega$ was not the frequency but the normalized period, i.e., $t \in [0, 2\pi \tau]$ should be the $x$-axis. JB <<>>

In dynamical systems, bifurcations are of key importance to understand global parameter dependence of the dynamics. The analysis of bifurcations by pen and paper is generally restricted to cases where the solution at the bifurcation point is known analytically, and even in such cases it is often not feasible to examine properly the associated eigenvalue problem by hand. 
Hence numerical methods are applied ubiquitously to study bifurcation diagrams,
for example using specialized software such as AUTO~\cite{soft-auto}, MatCont~\cite{soft-matcont}, PyDSTool~\cite{soft-pydstool},  
XPP~\cite{soft-xpp} and COCO~\cite{soft-coco}.
This involves both the numerical continuation of solutions as well as the computational analysis of bifurcation points. To make such numerical simulations into rigorous mathematical statements, additional effort is required. For this purpose rigorous verification schemes (sometimes referred to as {\em a posteriori error analysis}) have been developed for a variety of continuation problems in  the past decade, 
%see~\cite{MR3444942,gomez_survez,MR1420838,jay_konstantin_survey,MR1849323,Plu01,MR2652784,MR2807595} 
see~\cite{koch_arioli,breden-vanicat,GLPmulti,BLM,QJBcontinuation,BWbifdia} 
and the references therein. The analogous methodology for bifurcation problems is much less developed, although some foundational results on pitchfork and saddle-node have been obtained in~\cite{koch_arioli,Kanzawa_Oishi,MR3542954,MR3808252,MR3792794,MR2534406,MR3390404}. Moreover, double turning points \cite{MR2319947,MR2059468}, period doubling bifurcations \cite{MR2534406} and cocoon bifurcations  \cite{MR2351028} have also been considered. %A weaker (topological) notion of bifurcations for steady states of nonlinear partial differential equations is also combined in \cite{MR2788927} with rigorous numerics. 

In this paper we develop a general framework for the rigorous verification of Hopf bifurcations. We consider the class of polynomial vector fields
\begin{equation}\label{e:initial}
\dot{u}=f(u,\mu)
\end{equation}
where $u \in  \mathbb{R}^n$, $\mu \in \mathbb{R}$ represents the parameter in the system and $f:\R^n \times \R \to \R^n$ is polynomial both in $u$ and $\mu$.
%The appended $g: \R^{M'} \to \R^{M''}$ represents imposed relations between the parameters.  
In our presentation, we choose the parameter $\mu$ to be one-dimensional, as generically it is the appropriate condition for curves of periodic orbits to exist.
%\begin{equation}\label{e:initial}
%  \left\{ \begin{array}{l} 
%    \dot{u}=f(\mu,u) ,\\
%    g(\mu)=0,
%  \end{array} \right.
%\end{equation}
%%
%where $u(t) \in  \mathbb{R}^{N'}$, and $\mu \in \mathbb{R}^{M'}$ represents the parameters in the system. 
%The appended $g: \R^{M'} \to \R^{M''}$ represents imposed relations between the parameters.  
%A dimension count shows that generically $M'=M''+1$ is the appropriate condition for curves of periodic orbits to exist.
%Throughout the paper we will thus assume that $M''=M'-1$.
%We assume both $f$ and $g$ to be polynomial.
The polynomial dependence of $f$ on $u$ means that the estimates from~\cite{QJBcontinuation} apply. We note that~\cite{QJBcontinuation} includes an extensive discussion of the technical advantages that polynomial vector fields offer, as well as an array of generalizations.
% The main reason to restrict attention to polynomial dependence of $f$ on the parameter $\mu$ is that for the implementation we build on the rigorous
% continuation code developed in~\cite{QJBcontinuation}.
Indeed, the ideas in the current paper can be carried through for nonpolynomial vector fields, but that will require some supplemental effort. For example, by introducing new variables it is possible to transform a nonpolynomial vector field into a higher dimensional polynomial vector field (e.g. see \cite{MR3770054,MR633878,MR3545977}), and then apply a slight modification of the approach proposed in the current paper. 
Additionally, the presented approach can be extended to infinite dimensional dynamical systems described by delay-differential equations or parabolic partial differential equations, which is work in progress.

While our primary aim is the study of Hopf bifurcations, along the way we develop a technique to rigorously establish non-degeneracy of fold bifurcations for periodic orbits. Indeed, using a blowup strategy we convert the Hopf bifurcation problem into a regular continuation problem. A simpler version of this desingularizaton technique was already used in computationally analyzing the periodic solutions near the Hopf bifurcation in Wright's delay equation in~\cite{Ldelay}, see also~\cite{JBJono} for a similar but essentially analytic version.
A fold in the associated continuation problem corresponds to a Hopf bifurcation in the original bifurcation problem.

% \corrc I added a bit more below to address the second question of JB at the end of the Introduction. JP <<>>
%

The main contribution of this paper is a flexible and mathematically rigorous computational framework to study folds of periodic orbits and Hopf bifurcations in ODEs. In particular, the blowup technique allows computing a smooth global branch of periodic orbits starting from a Hopf bifurcation point. 
In future work we plan to adapt the blowup technique to other symmetry breaking bifurcations (e.g. pitchfork and period-doubling) yielding rigorous computations of global smooth branches of periodic orbits starting from such bifurcation points.

We now introduce the main ingredients of the paper, referring to subsequent sections for precise statements.
Since in general the period $L$ of the periodic solutions is a priori unknown and depends on the parameters, it is convenient to rescale time and add the (normalized) period $\tau = \frac{L}{2\pi}$  to the set of parameters $\lambda=(\tau,\mu) \in \R^{2}$ to arrive at the system
\begin{equation}\label{e:prelim}
  \left\{ \begin{array}{l} \dot{u} = \tilde{f}(u,\lambda),\\
%  g(\lambda)=0 \\
  u \text{ is $2\pi$-periodic},
  \end{array}\right.
\end{equation}
where the derivative is now with respect to the new time variable, and
$\tilde{f}(u,\lambda) \bydef \tau f(u,\mu)$.
%and we write $g(\lambda)=g(\mu)$.

We break the translation invariance of solutions to the autonomous problem by appending a phase condition. 
Moreover, to describe curves of solutions, it is expedient to introduce a ``continuation'' equation, depending on a continuation parameter $s$. 
Since we will work primarily in Fourier space, it is convenient to choose a phase condition which depends on the Fourier modes rather than coordinates in phase space. Hence we introduce $\hat{u}$ to denote the Fourier coefficients of $u$. 
We note that without any essential loss of flexibility we restrict attention to 
phase and continuation equations that depend affine \emph{linearly} on \emph{finitely} many of the Fourier modes $\hat{u}$ and $\lambda$ only, while also the dependence on~$s$ may be affine linear.
This leads to a system
\begin{equation}\label{e:main}
  \left\{ \begin{array}{l} \dot{u} = \tilde{f}(u,\lambda), \\
  \tilde{g}_s(\hat{u},\lambda)=0, \\
  u \text{ is $2\pi$-periodic},
  \end{array}\right.
\end{equation}
where $\tilde{g}_s \in \R^{2}$ represents the two (phase and continuation) appended equations just discussed.
We expect to find curves of solutions to~\eqref{e:main}, parametrized by $s$, where both $u(t)$ and $\lambda=(\tau,\mu)$ are unknowns. In what follows, some additional ``phase'' conditions (to be specified later) will be absorbed into~$\tilde{g}_s$, for example in order to break natural continuous symmetries. 

For fixed $s$ the problem~\eqref{e:main} is expected to have isolated solutions,
and we may move to Fourier space to set up a corresponding fixed point map. We do this in Section~\ref{s:setup} and we reduce checking contractivity of this map in a ball of radius $r$ (in some appropriately chosen Banach space) around a numerical approximation of a solution, to checking finitely many inequalities.
All the bounds, parametrized by~$r$, necessary to verify contractivity have been formulated in great generality in~\cite{QJBcontinuation}.  The verification of contraction can then be carried out with a computer based on interval arithmetic calculations. To obtain a curve of solutions we apply the uniform contraction principle, with explicit error bound given by the smallest $r=r_{\min}$ for which we can prove (uniform) contractivity. This parametrized Newton-Kantorovich methodology is explained in more detail in Section~\ref{s:setup}.

Given a bounded interval $I \subset \R$, to check that the solution curve $\{(u(t;s),\lambda(s))\}_{s \in I}$ (with $\lambda(s)=(\tau(s),\mu(s))$) has a nondegenerate fold (or saddle-node) bifurcation with respect to $\mu$ we need to find an $\sstar\in I$ such that $\mu'(\sstar)=0 $ and $\mu''(\sstar) \neq 0 $. 
To find the values of the $s$-derivative, we 
differentiate~\eqref{e:main} twice with respect to~$s$, and consider the derivatives to be part of the set of unknowns. We arrive at an extended system
\begin{equation}\label{e:mainextended}
  \left\{ \begin{array}{l} \dot{\bs{u}} = \bs{\tilde{f}}(\bs{u},\bs{\lambda}), \\
  \bs{\tilde{g}}_s(\hat{\bs{u}},\bs{\lambda})=0, \\
  \bs{u} \text{ is $2\pi$-periodic},
  \end{array}\right.
\end{equation}
for $\bs{u}(t) \bydef (u,u',u'')(t) \in \R^{3n}$ and
$\bs{\lambda} \bydef  (\lambda,\lambda',\lambda'') \in \R^{6}$, where primes denote
derivatives with respect to~$s$. The extended $\bs{\tilde{f}}$ represents a vector
field on $\R^{3n}$, and $\bs{\tilde{g}}_s(\hat{\bs{u}},\bs{\lambda}) \in
\R^{6}$ incorporates 
%both algebraic equations (depending on $\bs{\lambda}$ only) and 
extended phase and continuation equations which depend affine
linearly on finitely many of the Fourier modes $\hat{\bs{u}}$ and~$\bs{\lambda}$, as well as affine linearly on $s$. We conclude
that~\eqref{e:mainextended} is thus of the same form as~\eqref{e:main}, hence
the continuation method from Section~\ref{s:setup} and~\cite{QJBcontinuation} still applies.
We then use the rigorous error control $r_{\min}$ to check that for some interval $[s_0,s_1]$ we have 
\[
\left\{
\begin{array}{l}
  \mu'(s_0) \,\mu'(s_1) < 0 \\[2mm]
  |\mu''(s)| >0, \qquad \text{for all }  s\in [s_0,s_1].
\end{array}
\right.
\]
This guarantees the existence of a single, nondegenerate fold bifurcation at some $\sstar \in (s_0,s_1)$ along the curve $\{(u(t;s),\lambda(s))\}_{s \in [s_0,s_1]}$. The sign of the second derivative controls the direction of the fold. The details of this construction are presented in Section~\ref{s:saddlenode}.
In Section~\ref{s:eigenvalues} we discuss the information which we can extract about the eigenvalue behaviour (and thus (in)stability of the periodic orbits).

To capture periodic orbits that bifurcate from an equilibrium at a Hopf bifurcation, we use a blowup procedure. We write $u(t)=\y+a \bar{u}(t)$, where $\y \in \R^n$ solves the equilibrium version of~\eqref{e:main}, and the ``amplitude''
$a \in \R$ is appended to the set of parameters, while simultaneously an additional ``amplitude'' condition, say $g^{\ampl}_s(\hat{\bar{u}})=0 \in \R$, is imposed, which embodies that $\bar{u}$ is order 1. The amplitude condition function $g^{\ampl}_s$ depends again affine linearly on finitely many Fourier modes.
%  \blue{Does it really depend on $\hat{\bar{u}}$ only, or also on some of the parameters? Like $w$?}\red{To be precise, the amplitude function is
% $$
% \|\hat{\bar{u}} \|^2= 1
% $$
% and does not depend on any extra parameter.
% }

Since $\y \in \R^n$ is a time-independent equilibrium solution, it also gets appended to the set of parameters:
\[
   \bar{\lambda} \bydef (\tau,a,\y,\mu) \in \R^{3+n},
\]
and correspondingly 
\[
  \bar{g}_s(\hat{\bar{u}},\bar\lambda) 
% = (\tilde{g}^0_s(\lambda,\hat{\bar{u}}),g^{\ampl}_s(\hat{\bar{u}}), \tilde{f}(w,\lambda) )\in \R^{3+n},
\bydef ( \tilde{g}_s^{\text{phase}}(\hat{\bar{u}}),g^{\ampl}_s(\hat{\bar{u}}), f(y,\mu),\tilde{g}_s^{\text{cont}}(\hat{\bar{u}},\bar{\lambda}) )\in \R^{3+n},
\]
where 
$\tilde{g}_s^{\text{phase}}=0$ and  $\tilde{g}_s^{\text{cont}}=0$ denote the phase condition and continuation equation, respectively. 
% $\tilde{g}_s^{\phase}(\lambda,\hat{\bar{u}})=0$ denotes the phase condition and $\tilde{g}_s^{\cont}(\bar{\lambda},\hat{\bar{u}})=0$ denotes the continuation equation.

%$\tilde{g}^0_s (\bar{\lambda},\hat{\bar{u}})= \tilde{g}_s(\lambda,\hat{\bar{u}}) \in \R^2$, except for the continuation equation, which generally involves all elements of $\bar{\lambda}$ (as well as $\hat{\bar{u}}$).
% {Note to self: cannot quite work, since phase condition would disappear at $a=0$}\red{
% In practice, $\tilde g_s(\lambda, x) = \langle \imag\boldsymbol{k} x, x \rangle$, that in our case becomes
% $$
% \langle \imag\textbf{k} (w+ax),(w+a x) \rangle = \langle \imag\boldsymbol{k} (ax),a x \rangle
% $$
% because the 0th Fourier mode is multiplied by its index, and is simplified to
% $$
% \langle \imag\boldsymbol{k} u, u \rangle = 0.
% $$
% Indeed, without simplification it wouldn't work for $a = 0$.
% }
The ODE for $\bar{u}(t) \in \R^{n}$ is given by
\begin{equation} \label{eq:bar_f_defn}
  \dot{\bar{u}} = \bar{f}(\bar{u},\bar{\lambda}) \bydef
  \begin{cases}
	   \displaystyle \frac{\tilde{f}(\y+a \bar{u},\lambda)-\tilde{f}(\y,\lambda) }{a}  & \text{if } a \neq 0,
	   \vspace{.2cm}	   
	   \\
	   D_u \tilde{f}(\y,\lambda)\bar{u}   & \text{if } a =0,
\end{cases}	   
\end{equation}
which represents a rescaled smooth vector field, and $\bar{f}$ is polynomial if $\tilde{f}$ is.
The new \emph{desingularized Hopf problem}
\begin{equation}\label{e:mainhopf}
  \left\{ \begin{array}{l} \dot{\bar{u}} = \bar{f}(\bar{u},\bar{\lambda}), \\
  \bar{g}_s(\hat{\bar{u}},\bar{\lambda})=0, \\
  \bar{u} \text{ is $2\pi$-periodic},
  \end{array}\right.
\end{equation}
is then again of the form~\eqref{e:main}, hence the continuation machinery from Section~\ref{s:setup} and~\cite{QJBcontinuation} is directly applicable. This leads to continuation of periodic solution ``through'' the Hopf bifurcation at $a=0$.
In order to show that the Hopf bifurcation is nondegenerate, and to determine its direction, the saddle-node construction~\eqref{e:mainextended} is then applied to~\eqref{e:mainhopf}. More precisely, this leads to the extended system
\begin{equation}\label{e:mainhopfextended}
  \left\{ \begin{array}{l} \dot{\bar{\bs{u}}} = \bs{\bar{f}}(\bar{\bs{u}},\bs{\bar{\lambda}}), \\
  \bs{\bar{g}}_s(\hat{\bar{\bs{u}}},\bs{\bar{\lambda}})=0, \\
  \bar{\bs{u}} \text{ is $2\pi$-periodic},
  \end{array}\right.
\end{equation}
for $\bar{\bs{u}}(t)=(\bar{u},\bar{u}',\bar{u}'')(t) \in \R^{3n}$ and
$\bs{\bar{\lambda}}=(\bar{\lambda},\bar{\lambda}',\bar{\lambda}'') \in \R^{3(3+n)}$, where primes denote
derivatives with respect to $s$. The details of this construction are given in Section~\ref{s:hopf}.
When continuing a branch of periodic orbits that originates from a Hopf bifurcation far away from the bifurcation point, one would like to switch back from the desingularized formulation~\eqref{e:mainhopf} to the original system~\eqref{e:prelim}. This topic is discussed briefly in Section~\ref{s:gluing}.

% In the second case, before being able to applying the gluing, we need to precompute the fixed point around which the orbit is circling, i.e. the fixed point undergoing the Hopf bifurcation. This can easily be achieved with a Newton method, using as starting point the geometrical center of the orbit.

\begin{figure}[t]
\centering
\includegraphics[width=\textwidth]{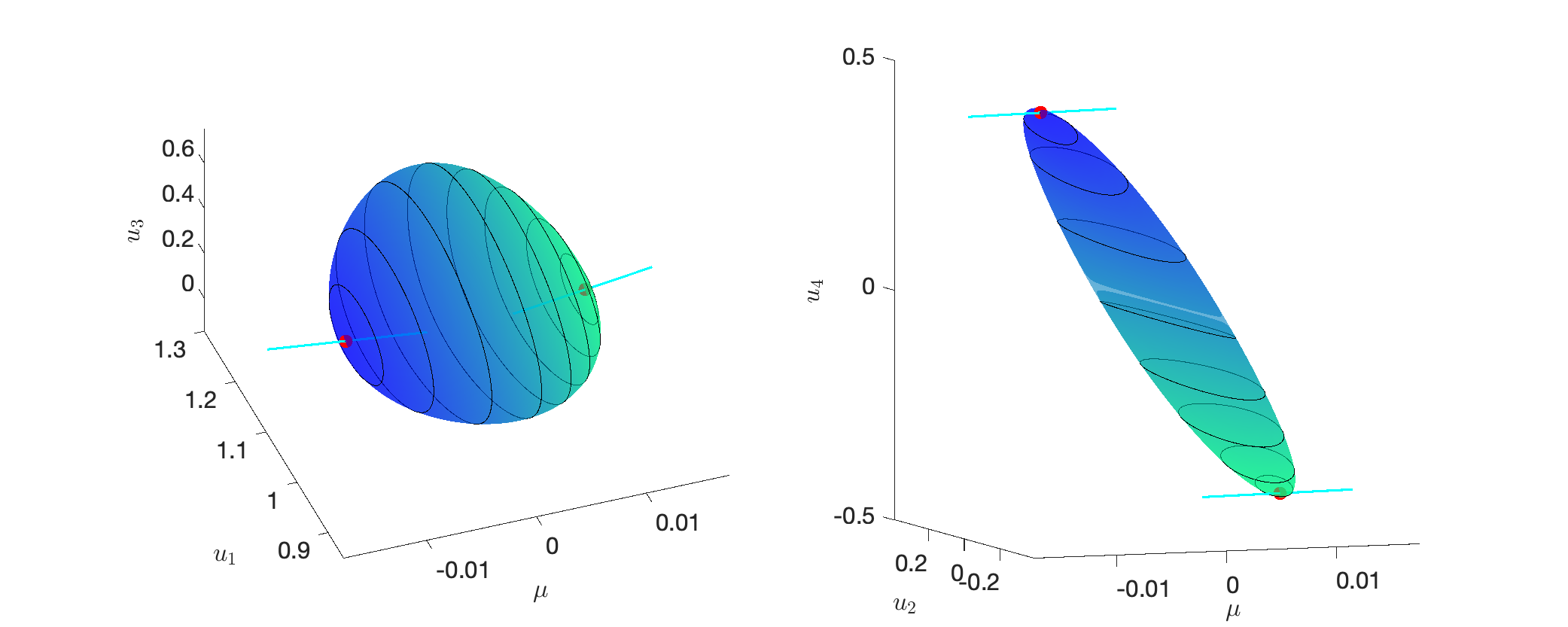}
\caption{A branch of periodic solutions of system \eqref{lorenz84intro} with varying $\mu$ and fixed parameters $(\xi_1,\xi_2,\xi_3,\xi_4,\xi_5,\xi_6)=(0.25,0.987,1,0.25,1.04,2)$. 
Validated continuation combined with gluing (see Section~\ref{s:gluing}) proves the existence of such a \emph{balloon} of periodic orbits, connecting two equilibria (red dots) undergoing the Hopf bifurcations at different parameter values. These two equilibria are not connected by continuation at the level of stationary states, nor related by symmetry, see Section~\ref{s:lorenz84} for more details. The branches of the equilibria undergoing the Hopf bifurcations are plotted in cyan. The balloon is colored by $\mu$-value.
}\label{f:lorHopf2Hopf}
\end{figure}

To illustrate our approach with an example, consider the extended Lorenz-84 system~\cite{KuznetsovMeijer,KuznetsovMeijervanVeen}
\begin{equation}\label{lorenz84intro}
	\begin{cases}
	\dot u_1 = -u_2^2-u_3^2 - \xi_1 u_1 - \xi_1 \xi_6 - \xi_2 u_4^2,\\
	\dot u_2 = u_1u_2 - \xi_3 u_1u_3 - u_2 + \xi_4,\\
	\dot u_3 = \xi_3 u_1u_2 + u_1u_3 - u_3,\\
	\dot u_4 = -\xi_5 u_4 + \xi_2 u_4u_1 + \mu,
	\end{cases}
\end{equation}
which has a four dimensional phase space and seven parameters.
In Figure \ref{f:lorHopf2Hopf} we depict a full continuous branch of periodic orbits of \eqref{lorenz84intro} with $\mu$ as the ``continuation'' parameter,  starting at one Hopf bifurcation point and finishing at another one. The two equilibria with Hopf bifurcation points, which are connected by this branch of periodic solutions, lie on different continuation curves, i.e., the periodic orbits ``cross over'' from one branch of equilibria to another. 
Further details are discussed in Section~\ref{s:lorenz84}.

Several other examples of fold and Hopf bifurcations are also presented in Section~\ref{s:examples}. This includes an illustration in Section~\ref{s:hamiltonianexample} of how Hamiltonian systems, which are very nongeneric from the point of view of periodic orbits, can nevertheless be analyzed using the methods in the current paper.
The accompanying \textsc{matlab} code can be found at~\cite{codehopf}.

% \corrc Not sure if we should give more details about eigenvalue information in the introduction. JB <<>>

%\corrc We (JB and JP) talked about adding a paragraph here explaining the significance of this paper. We really take a radically different point of view compared to traditional papers which focus on the bifurcation \emph{point}. <<>>

%\corrc Outline of the paper to be added. JB <<>>

% \corrc Not sure if we should write more about why all this is interesting. JB <<>>

% \blue{What is still missing are remarks concerning the non-degeneracy of the Hopf bifurcation.}  \blue{Also, perhaps eigenvalue information and remarks about sub/supercritical can be added.}
% \blue{Note also that in this methodology we do not check whether possible other pairs of eigenvalues cross the in addition to the one responsible for the creation of the periodic orbit that we are curve following.}{\color{green} This can be taken care of by solving the algebraic problem}

\section{Setup for the continuation of periodic orbits in Fourier space}
\label{s:setup}
%!TEX root = bifcont.tex

We briefly introduce the setup of the continuation problems under consideration in this paper. Additional details can be found in~\cite{QJBcontinuation}. Consider a polynomial vector field of the form 
\begin{equation} \label{eq:h_vector_field}
  \left\{ \begin{array}{l} \dot{u} = h(u,\lambda),\\
 g(\lambda)=0, \\
  u \text{ is $2\pi$-periodic},
  \end{array}\right.
\end{equation}
where $h:\R^n \times \R^m \to \R^n$ is a polynomial vector field and $g:\R^m \to \R^{m'}$ is a polynomial mapping. The maps $h$ and $g$ depend on the problem under study. For the standard \emph{pseudo-arclength} continuation of periodic orbits, $m=2$, $m'=0$, $\lambda=(\tau,\mu) \in \R^2$, $h=\tilde{f}:\R^n \times \R^{2} \to \R^{n}$ is defined in \eqref{e:prelim} and $g \equiv 0$. For the continuation of periodic orbits passing through a Hopf bifurcation, $m=3+n$, $m'=n$, $\lambda=(\tau, a, \y,\mu) \in \R^{3+n}$, $h=\bar{f}:\R^{n} \times \R^{3+n} \to \R^{n}$ is defined in \eqref{eq:bar_f_defn} and $g(\lambda)=f(\y,\mu)$.

\begin{remark}
For the pseudo-arclength continuation, $g \equiv 0$ and therefore two equations are ``missing'' to balance the variable $\lambda \in \R^2$. In this case a phase condition and a continuation equation are appended. For the desingularized Hopf problem, three equations are missing as $\lambda \in \R^{3+n}$ is variable and $g(\lambda)=0 \in \R^n$. A phase condition, a continuation equation and an amplitude equation are then introduced.
\end{remark}

\subsection{Formulation in Fourier space}
\label{s:fourier}

We write the Fourier expansion of a $2\pi$-periodic function $u =(u_1,\dots,u_n) : \R \to \R^n$ as
\begin{equation}\label{e:fourier}
u(t)=\sum_{k\in\mathbb{Z}}(v)_{k} e^{\imag k t }, \qquad (v)_{k}=(v_1,\dots,v_n)_k\in \mathbb{C}^n.
\end{equation}
The differential equation $\dot{u} = h(u,\lambda) \in \R^n$ then transforms in Fourier space to
\begin{equation}\label{e:defF}
\left(F_i(x)\right)_k \bydef \imag k  (v_i)_{k} -
  (\widehat{h}_i (v,\lambda))_k =0 \qquad \text{for } 1\leq i \leq n, k\in \Z,
\end{equation}
where $\widehat{h}_i$ is the polynomial $h_i$, but with multiplications interpreted as convolutions, denoted
\[
(\tilde{v} \tilde{v}')_k \bydef \sum_{k' \in \Z}  \tilde{v}_{k'} \tilde{v}'_{k-k'}.
\]

To fix a Banach space in which we will apply contraction arguments, 
we introduce the $\nu$-norm ($\nu \geq 1$) on $\C^\Z$ as
\begin{equation}\label{norm_l1nu}
\|\tilde{v}\|_\nu \bydef \sum_{k \in \Z}  |\tilde{v}_k| \nu^{|k|} ,
\end{equation}
with corresponding Banach space 
$\ell^1_\nu \bydef \{ \tilde{v} \in \C^\Z \,:\, \|\tilde{v}\| _\nu <\infty\}$.
The space of variables $x=(v,\lambda)$ is then $X=X_\nu \bydef (\ell^1_\nu)^n \times \C^m$.
On the components of $x=(x_1,\dots,x_{m+n})=(v_1,\dots,v_n,\lambda_1,\dots,\lambda_m)$
we define the norm
\[
  \|x_j\| \bydef \begin{cases}
  \|x_j\|_\nu \qquad &\text{for }1\leq j\leq n,\\
  |x_j| & \text{for }n+1\leq j \leq n+m,
  \end{cases}
\]
leading to the product norm
\begin{equation}\label{e:maxnorm}
	\|x\|_X \bydef \max_{1\leq j \leq m+n} \| x_j\|.
\end{equation}	
We introduce the time derivative on $\ell^1_\nu$ by
\[
  (\imag \K  \tilde{v})_k \bydef \imag k \tilde{v}_k, 
  \qquad \text{for } k \in \Z,
\]
which we extend to $X$ via $\imag \K x = (\imag \K v_1,\dots,\imag \K v_n,0,\dots,0)$.

Finite dimensional numerical approximations are found in the truncated space
\[
  X_K \bydef \{ x=(v,\lambda) \in X : (v_i)_k=0 \text{ for all } |k|>K, 1 \leq i \leq n \}
\]
for some $K \in \N$.
The space $X_K$ can be identified with $\C^{m+n(2K+1)}$ and it will be convenient notationally to introduce the following bilinear form on $X_K$:
\begin{equation}\label{e:bilinear}
  \langle x, x' \rangle \bydef \sum_{i=1}^{n} \sum_{k=-K}^K (v_i)_k (v'_i)_k  
  + \sum_{j=1}^m \lambda_j \lambda'_j .
\end{equation}

In order to recover a real-valued solution we will check a posteriori that  $(v_i)_{-k} = \overline{(v_i)_k}$ using equivariance of the problem under the conjugation symmetry defined below.
\begin{definition}
The \emph{conjugate} $x^* = (v^*,\lambda^*)$ of $x=(v,\lambda) \in X$ is given by 
\[
  \lambda^*_j = \overline{\lambda_j} ~~\text{for } 1\leq j \leq m, 
  \qquad 
  (v^*_i)_k = \overline{(v_i)_{-k}} ~~\text{for } 1\leq i \leq n, k \in \Z.
\]
The set of conjugate symmetric elements is denoted by $\XS = \{ x \in X : x^*=x\}$, and $\XS_K = \XS \cap X_K$.
\end{definition}
We note that both $F$ and $g$ are equivariant under the conjugation symmetry: $F(x^*)=F(x)^*$ and $g(\lambda^*)=g(\lambda)^*$. 
Note that the equivariance of  $g$ follows from the fact that it is a real polynomial mapping. 
Furthermore,
$ \langle x, x' \rangle \in \R$ for $x, x' \in \XS_K$.

To set up the rigorous continuation framework, let us assume we have two points $\hx_0 = (\hv_0,\hl_0) \in \XS_K$ and $\hx_1 = (\hv_1,\hl_1) \in \XS_K$ which each represent an approximate solution
of
\begin{equation}\label{e:Fgsystem}
  \left\{
  \begin{array}{l}
	  F(x)=0, \\
	  g(\lambda)=0 .
  \end{array}
  \right.
\end{equation}
%where $g(\lambda)$ represents the set of algebraic constraints.
We define the interpolation
\[
  \hx_s = (\hv_s, \hl_s)  \bydef (1-s) \hx_0 + s\hx_1
  \qquad\text{for  } s\in [0,1].
\]
To introduce the phase condition we define 
$q^\phase = (q^\phase_v,0) \in \XS_K$
with 
\begin{equation}\label{e:choiceofphaseq}
  (q^\phase_v)_j  =  \overline{ \imag \K (\hv_{\frac{1}{2}})_j} \qquad\text{for }
  |k| \leq K, 1\leq j \leq n.
\end{equation}
We choose the phase condition
\[
  G^\phase(x) \bydef   
  \bigl\langle q^\phase , x \bigr\rangle  
  % - \bigl[(1-s)
  %   \langle \hx_0 , \hq_0 \rangle + s \langle \hx_1 , \hq_1 \rangle \bigr]
  =0 .
\] 
This phase equation does not depend on $s$, which deviates slightly from the one in~\cite{QJBcontinuation}. Having an $s$-dependent phase condition is convenient when validating  long stretches of a solution branch (essentially, it assists in ``gluing'' short pieces into a long smooth curve). Here we are interesting in verifying a solution branch near a bifurcation point, hence $s$-dependence is not necessary and merely complicates the algebra and notation.
% $\hq_\s = (q_\s,0) \in \XS_K$ for $\s=0,1$
% with
% \begin{equation}\label{e:choiceofphaseq}
%   ((q_\s)_i)_k = \imag k ((\hv_\s)_i)_{-k} \qquad\text{for }
%   |k| \leq K, 1\leq i \leq n.
% \end{equation}
% We interpolate these as $\hq_s = (1-s) \hq_0 + s \hq_1$.
% We choose the phase condition
% \[
%   G^\phase_s(x) \bydef
%   \bigl\langle x , \hq_s \bigr\rangle
%   % - \bigl[(1-s)
%   %   \langle \hx_0 , \hq_0 \rangle + s \langle \hx_1 , \hq_1 \rangle \bigr]
%   =0 .
% \]
%\corrc I wrote the phase condition in a simpler way, since $\langle \hx_0 , \hq_0 \rangle$ in fact vanishes, I think. As a consequence I slightly changed the wording at the start of the remark below. JB <<>> 
\begin{remark}\label{r:dependence}
We note that $G^\phase$ depends linearly on $x$.
%as well as affine linearly on $s$. 
Furthermore, $G^\phase$ only depends on the Fourier coefficients with indices $|k| \leq K$. We will encounter slightly more general dependence for more general ``phase'' equations throughout, namely we allow for affine linear dependence on  $x$ and affine linear dependence on $s$. For this purpose we introduce the notation
\begin{equation}\label{e:generalG}
  \G^{\pp_0,\pp_1}_{\yy_0,\yy_1}(x,s) \bydef 
  \bigl\langle (1-s) \pp_0 + s \pp_1 , x \bigr\rangle  - \bigl[(1-s)
    \yy_0  + s \yy_1  \bigr],
\end{equation}
for $\pp_0,\pp_1\in \mathcal{S}_K$ and $\yy_0,\yy_1 \in \R$. 
In this notation $G^\phase(x)=\G^{q^\phase,q^\phase}_{0,0}(x,s)$.
%{\langle \hx_0 , \hq_0 \rangle,\langle\hx_1 , \hq_1 \rangle }(x,s)$.
We note that \[ 
\G^{\pp_0,\pp_1}_{\yy_0,\yy_1}(x^*,s)=  \overline{\G^{\pp_0,\pp_1}_{\yy_0,\yy_1}(x,s)}.\]
\end{remark} 

\begin{remark}
%\corrc The amplitude equation has also been changed. JB <<>>
In case $h=\bar{f}:\R^{3+n} \times \R^{n} \to \R^{n}$ is the desingularized Hopf problem defined in \eqref{eq:bar_f_defn}, an extra phase-like condition is appended to the system, namely the ($s$-independent) amplitude equation
\[
G^\ameq(x) \bydef   \bigl\langle q^\ameq , x \bigr\rangle  - 1 =   \G^{q^\ameq,q^\ameq}_{1,1}(x,s)  =0,
\]
where 
$q^\ameq = (q^\ameq_v,0) \in \XS_K$,
with 
%\corrc Note the new $\K^2$. Elena: has it been incorporated in the code? JB <<>>
%\corrc yes. Elena <<>>
\[
  (q^\ameq_v)_i =  \overline{\K^2 (\hv_{\frac{1}{2}})_i} \qquad\text{for }
  |k| \leq K, 1\leq i \leq n.
\]
See Section~\ref{s:hopf} for more details and a motivation for this choice. 
\end{remark}
% \begin{remark}\red{REDO}
% In case $h=\bar{f}:\R^{3+n} \times \R^{n} \to \R^{n}$ is the desingularized Hopf problem defined in \eqref{eq:bar_f_defn}, an extra phase-like condition is appended to the system, namely the amplitude equation
% \[
% G^\ameq_s(x)=\G^{\hp_0,\hp_1}_{1,1}(x,s),
% \]
% %
% where
% $\hp_\s = (0, p_\s) \in \XS_K$ for $\s=0,1$,
% with $p_\s$ given by
% \[
%   ((p_\s)_i)_k = ((\hv_\s)_i)_{-k} \qquad\text{for }
%   |k| \leq K, 1\leq i \leq n.
% \]
% See Section~\ref{s:hopf} for more details.
% \end{remark}

%\corrc Fixed the use of predictors, which was missing a complex conjugate. JB <<>>
In an analogous manner we introduce an $s$-dependent continuation equation $G^\cont_s(x)=0$. Indeed, we determine numerically ``predictors'' $\dot{\hx}_\s \in \XS_K$ of the tangent direction of the solution curve (for the problem including the phase condition) at $\hx_\s$ for $\s=0,1$. 
We set 
\[
  q^\cont_{\s} = \overline{\dot{\hx}_{\s}} ,
\]
%$(\dot{\hq}_\s)_k = \overline{(\dot{\hx}_\s)_k}$, 
so that $q^\cont_{\s} \in \XS_K$, and we define 
\begin{align}
G^\cont_s(x) &\bydef \bigl\langle (1-s) q^\cont_0 + s q^\cont_1 , x \bigr\rangle - \bigl[ (1-s) \langle  q^\cont_0 , \hx_0 \rangle + s \langle q^\cont_1 , \hx_1 \rangle \bigr] \nonumber \\
&= \G^{q^\cont_0,q^\cont_1}_{\langle  q^\cont_0 , \hx_0  \rangle ,\langle  q^\cont_1 ,  \hx_1 \rangle }(x,s) \label{e:continuationequation},
\end{align}
hence the dependence of $G^\cont$ on $x$ and $s$ is  as described in Remark~\ref{r:dependence}. 

The full set of ``algebraic'' equations is
\begin{equation*}%\label{e:defG_pal_cont}
G_s(x) =
\begin{bmatrix}  G^\phase(x)  \\ G^\cont_s(x) \end{bmatrix} 
% \end{equation}
% %
% the pseudo-arclength continuation,
% %
% \begin{equation}\label{e:defG_hopf}
\qquad\text{or}\qquad
G_s(x) =
\begin{bmatrix}  G^\phase(x) \\ G^\ameq(x) \\ g(\lambda) \\G^\cont_s(x)  \end{bmatrix} ,
\end{equation*}
for the pseudo-arclength continuation and the desingularized Hopf problem, respectively. The general zero finding problem for continuation is 
\begin{equation}\label{e:defHs}
H_s(x) \bydef\begin{bmatrix} F(x) \\ G_s(x) \end{bmatrix} = 0 .
\end{equation}
Clearly, conjugate symmetric zeros of $H_s$ correspond to periodic orbits of~\eqref{e:initial}, provided $\tau=\lambda_1 \neq 0$.

We define the associated fixed point operator
\begin{equation}\label{e:defTs}
T_s(x) \bydef x-A_s H_s(x), \qquad T_s:X \to X, \qquad s\in [0,1].
\end{equation}
Here $A_s$ is an injective map that approximates the inverse of the Jacobian $D_x H_s(\hx_s)$. We do not elaborate on the choice for $A_s$, which is discussed in detail in~\cite[Section~8.2]{QJBcontinuation}. 
For the current discussion it suffices to say that $A_s=(1-s)A_0 + s A_1$,
and $A_{\s}$,  $\s=0,1$ are approximate inverses of the Jacobians at the end points $\hx_\s$. Each linear operator $A_{\s}$, $\s=0,1$ is
made up from a $(m+n(2K+1))\times(m+n(2K+1))$ matrix and a diagonal infinite tail.
In particular, let $\Pi_K$ denote the natural projection of $X$ onto $X_K$, then
the block structure (finite matrix and infinite tail) of $A_{\s}$ is characterized by
%\corrc Shouldn't we replace $A$ by $A_s$ in what follows? JP  Yes, done. JB<<>>
$A_{\s} \Pi_K = \Pi_K A_{\s}$ and $ A_{\s} (I -\Pi_K) =  (I -\Pi_K) A_{\s}$,
while the diagonal tail is given by
\[
(I -\Pi_K) A_{\s} x = (I-\Pi_K)  (-\imag \K^{-1} v_1, \dots ,-\imag \K^{-1} v_n,
0,\dots , 0),
\] 
with $(\K^{-1} \tilde{v})_k \bydef k^{-1} \tilde{v}_k$ for any $k \neq 0$.
With regards to conjugation symmetry, the choice of $A_{\s}$ is such that
$A_{\s} x^* = (A_{\s} x)^*$, hence $T_s(x^*)=T_s(x)^*$ for all $s\in[0,1]$. 

Let $B_r(x) \bydef \{ x' \in X : \|x-x'\|_X \leq r \}$, then the ``tube'' around the numerical line segment $\{\hx_s: s\in [0,1]\}$
is given by
\begin{equation}\label{e:tube}
  \mathcal{C}_r \bydef \bigcup_{s \in [0,1]} B_r(\hx_s).
\end{equation}
In~\cite[Sections~6 and~8]{QJBcontinuation} explicitly computable bounds 
$Y=(Y_1,\dots, Y_{m+n})$ and $Z(r)=(Z_1,\dots,Z_{m+n})(r)$ are derived, that satisfy
\begin{subequations}
\label{e:boundsYZ}
\begin{alignat}{1}
Y_j &\geq \max_{s\in [0,1]}\| (T_s(\hx_s) - \hx_s )_j \| ,\\
Z_j(r) &\geq \max_{s\in [0,1]}\sup_{b, c\in B_1(0)}\|[  D_x T_s(\hx_s+r b)r c]_j \|,
\label{e:boundZ}
\end{alignat}
\end{subequations}
for $j=1,\dots,m+n$.
The following theorem, which is itself based on the uniform contraction principle (see for example~\cite{BLM,breden-vanicat,GLPmulti} for similar results), is the crux of rigorously verified continuation.
\begin{theorem}[Theorems~3.1~and~4.2 in~\cite{QJBcontinuation}]\label{thm:radii_bound}
Assume $Y$ and $Z(r)$ satisfy~\eqref{e:boundsYZ}.
Assume moreover that $A_s$ is injective for all $s\in[0,1]$. 
If there exists an $\hr>0$ such that 
\begin{equation}\label{e:radpoly}
  Y_j+Z_j(\hr)-\hr < 0 
\qquad \text{for all } j=1,\dots, m+n,
\end{equation}
then $T_s$ is a contraction on $B_{\hr}(\hx_s)$  for every $s$ in $[0,1]$. 
The fixed points $\sol{x}(s)$ of $T_s$ in $B_{\widehat{r}}(x_s)$
are conjugate symmetric and form 
a continuous parametrized curve $\sol{x}:[0,1]\rightarrow X$ in $\mathcal{C}_{\hr}$, such that $H_s(\sol{x}(s))=0$ for every $s\in [0,1]$.
\end{theorem}

\begin{remark}\label{r:injectivity}
Injectivity of $A_s$ follows by a computational check on the finite part, since invertibility of its diagonal tail is trivial to establish. For the implementation of the $Z$-bound chosen in \cite{QJBcontinuation}, the former computational check is in fact implied by $Z(\hr)<\hr$, see~\cite[Section~8.5]{QJBcontinuation}.
Since $\sol{x}(s) \in B_r(\hx_s)$, it follows from the inequalities~\eqref{e:radpoly} and the definition of the $Z$-bound~\eqref{e:boundZ} that
\[
  \| I - A_s D_x H_s(\sol{x}(s)) \|_{B(X)} <1,
\]
where $\| \cdot \|_{B(X)}$ is the bounded linear operator norm on $X$. Hence $D_x H_s(\sol{x}(s))$ is injective.
%  injectivity of $A_s$ implies that
% $D_x H_s(\sol{x}(s))$ is injective.
\end{remark}
%\corrc I think "$D_x H_s(\sol{x}(s))$ is injective" is correct, since injectivity of $A_s$ is already mentioned in the first sentence of the Remark. JB. I agree. JP <<>>

%\begin{itemize}
%\item reformulation of \eqref{basic} into $H(\lambda,u)$ - First draft done
%\item Fourier and the $\Xi$ space - First draft done
%\item continuation with continuation equation - First draft done
%\item reference to [other paper] for validation - First draft done
%\item 2 points system - First draft done
%\end{itemize}

\section{Saddle-node bifurcation}
\label{s:saddlenode}
%!TEX root = bifcont.tex

It is computationally straightforward to check that the parametrized solution curve $\{\sol{x}(s)\}_{s\in [0,1]}$ is smooth, for example based on~\cite[Lemma~3.3]{QJBcontinuation}, see also~\cite{breden-vanicat,BLM}. Indeed, we need to check (computationally, using interval arithmetic) that
\begin{equation}\label{e:smoothness}
  	\hat{r}  \, \langle  q^\cont_1 - q^\cont_0 , b \rangle 
  	\neq
    \langle  (1-s) q^\cont_1 + s q^\cont_0 , \hx_1 - \hx_0 
    \rangle  
	\quad \text{ for any } s \in [0,1] \text{ and } b\in B_1(0),
\end{equation}
which in practice is satisfied since $\dot{\hx}_0=\overline{q^\cont_0}$ and $\dot{\hx}_1=\overline{q^\cont_1}$ are both almost parallel to $\hx_1 - \hx_0$, and $\hat{r} \ll 1$.

From now on we assume that Inequality~\eqref{e:smoothness} is satisfied. 
%\corrc The following statement is not clear to me. Could we refer to some equations to make it clearer? JP <<>>
%\corrc I added an explanation. Enough  detail? JB <<>>
Since it implies that 
\begin{equation}\label{e:DfGsx}
	D_s G^\cont_s(x) \neq 0
	\qquad\text{for any} x \in \mathcal{C}_{\hr}, 
\end{equation}
we in particular obtain that $\sol{x}'(s) \neq 0$. 
Namely, let $\{x(s)=(1-s)\hx_0+s\hx_1+ \hat{r} b(s)\}_{s\in[0,1]}$ with $b(s) \in B_1(0)$ be any smooth parametrized curve solving $G^\cont_s(x(s))=0$. Then, by formally differentiating the latter identity, we arrive, after some rearrangement of terms, at
\[
    \langle  (1-s) q^\cont_1 + s q^\cont_0 , x'(s)
    \rangle =
  	 \hat{r}  \, \langle  q^\cont_0 - q^\cont_1 , b \rangle 
  	+
    \langle  (1-s) q^\cont_1 + s q^\cont_0 , \hx_1 - \hx_0 
    \rangle  .
\]
Therefore, Inequality~\eqref{e:smoothness} implies the derivative $\sol{x}'(s)$,
which exists since the curve is obtained through the uniform contraction principle (see e.g.~\cite{breden-vanicat,BLM}), cannot vanish, and the curve is smooth.
Since $G^\cont_s(\sol{x}(s))=0$ for all $s\in [0,1]$,
and the inequality~\eqref{e:DfGsx} implies that for any fixed $x\in \mathcal{C}_{\hr}$ the  function $[0,1] \ni s \mapsto G^\cont_s(x)$ vanishes at most once,  all points in  $\{x(s)\}_{s\in[0,1]}$ are distinct, hence the curve does not selfintersect (i.e.\ it is a smooth embedding).
 
%\corrc Am I happy not to include a uniqueness region? No. But restricting to linear dependence on $s$ in the continuation equation makes a useful formulation difficult. I think we either (re)introduce quadratic dependence or we just skip uniqueness in this paper. JB <<>>
 
%\corrc New remark to cut off the caps from the tubes, so that we can be precise about uniqueness. JB <<>>
% \begin{remark}\label{r:cylinder}
% Theorem~\ref{thm:radii_bound} combined with inequality~\eqref{e:smoothness}
% supplies a curve that parametrizes all solutions of~\eqref{e:Fgsystem} modulo
% phase shift (and modulo rescaling of the amplitude $a$ in the desingularized
% Hopf case, see Section~\ref{s:hopf}) in the cylinder
% $\widetilde{\mathcal{C}}_{\hr}$ resulting from removing the two caps from the
% tube $\mathcal{C}_{\hr}$ defined in~\eqref{e:tube}:
% \[
%   \widetilde{\mathcal{C}}_{\hr} \bydef (\mathcal{C}_{\hr} \cap \XS) \setminus (B^-_{\hr} \cup B^+_{\hr}),
% \]
% where the half-balls $B^{\pm}_{\hr}$, forming the caps of the tube, are given by
% \begin{alignat*}{1}
%   B^-_{\hr} &\bydef \{x \in B_{\hr}(\hx_0)  \cap \XS :
%   \langle  q^\cont_0  , x- \hx_0  \rangle <0 \}, \\
%   B^+_{\hr} &\bydef \{x \in B_{\hr}(\hx_1)  \cap \XS :
%   \langle q^\cont_1 , x- \hx_1   \rangle >0 \}.
% \end{alignat*}
% %\corrc Why this definition? unclear. Elena  Still unclear? JB Now clearer to me. Elena<<>>
% Here we have assumed that the orientation of the predictors $\dot{\hx}_{\s}$
% is chosen such that $\langle  q^\cont_\s , \hx_1 - \hx_0 \rangle >0$ for $\s=0,1$.
% \end{remark}

\begin{definition} \label{defn:saddle-node}
{\em
The solution curve has a \emph{nondegenerate fold bifurcation} with respect to some parameter $\lambda_j$ if there is an $\sstar \in (0,1)$ such that
\begin{equation}\label{e:nondegsadnod}
  \sol{\lambda}'_j(\sstar) = 0
  \qquad 
  \sol{\lambda}''_j(\sstar) \neq 0.  
\end{equation}
}
\end{definition}
%
%\corrc The above is now our definition of a nondegenerate fold bifurcation. Below I mention that we have no eigenvalue requirement. JB <<>>
In this section we explain how to establish such nondegenerate folds.
% The result will simultaneously provide a uniqueness region, see Remark~\ref{r:unique}.
Note that we do not impose any eigenvalue restrictions in the description~\eqref{e:nondegsadnod} of a nondegenerate fold. Considerations about eigenvalues and exchange of stability are discussed in Section~\ref{s:eigenvalues}.
%\corrc And another new remark. JB I like it. JP<<>>
\begin{remark}
We allow for any of the elements of the vector $\lambda$ to be interpreted as the bifurcation parameter. Of course, the obvious choice is to take the original parameter $\mu$ in~\eqref{e:initial} as the bifurcation parameter. In some cases (e.g.~Hamiltonian systems, boundary value problems) it may also be of interest to consider the normalized period $\tau$ as the bifurcation parameter, see Remark~\ref{r:hopftau} and the example in Section~\ref{s:hamiltonianexample}.
\end{remark}

To obtain equations for the derivatives $\sol{x}'(s)$ and $\sol{x}''(s)$, recall \eqref{e:defHs} and differentiate $H_s(\sol{x}(s))$ 
formally to obtain
\begin{subequations}
\label{e:three}
\begin{alignat}{1}
  H_s(\bs{x}^{[0]})	&=0, \label{e:three-a}\\
  D_x H_s(\bs{x}^{[0]}) \bs{x}^{[1]} + D_s H_s(\bs{x}^{[0]}) & =0 ,\label{e:three-b}\\
  D_x H_s(\bs{x}^{[0]}) \bs{x}^{[2]} +  2 D_s D_x H_s(\bs{x}^{[0]}) \bs{x}^{[1]} + D_x^2 H_s(\bs{x}^{[0]}) (\bs{x}^{[1]},\bs{x}^{[1]}) &=0 ,\label{e:three-c}
\end{alignat}
\end{subequations}
where we have used that $D_s^2 H_s(x)$ vanishes.
The system~\eqref{e:three} is solved by 
\[
(\bs{x}^{[0]},\bs{x}^{[1]},\bs{x}^{[2]})(s) = (\sol{x}(s),\sol{x}'(s),\sol{x}''(s)).
\] 

\begin{remark}\label{r:uniquederivatives}
When $\bs{x}^{[0]}=\sol{x}(s)$ solves~\eqref{e:three-a},
then the \emph{unique} solutions of~\eqref{e:three-b} and~\eqref{e:three-c}
are $(\bs{x}^{[1]},\bs{x}^{[2]}) = (\sol{x}'(s),\sol{x}''(s))$, 
provided $D_x H_s(\bs{x}^{[0]}(s))$
is an injective linear operator. Remark~\ref{r:injectivity} explains that this
injectivity holds whenever we have found our solutions through
Theorem~\ref{thm:radii_bound}, see also Remark~\ref{r:triangular}.
\end{remark}

We now show that the extended system is again of the general form~\eqref{e:defHs}, that is, finitely many algebraic equations and generalized phase equations with structure as described in Remark~\ref{r:dependence},
and a polynomial vector field in Fourier space variables.
Hence we can apply the construction of Theorem~\ref{thm:radii_bound} to find solutions of~\eqref{e:three}. 

We introduce $\bs{x}=(\bs{v},\bs{\lambda}) \in (\ell^1_\nu)^{3n} \times \C^{3m} = \bs{X} \equiv X^3$, which we also represent as
$\bs{x}=(\bs{x}^{[0]},\bs{x}^{[1]},\bs{x}^{[2]})$, with $\bs{x}^{[i]}=(\bs{v}^{[i]},\bs{\lambda}^{[i]}) \in X$ for $i=0,1,2$.
The extended vector field $\bs{h}(\bs{u},\bs{\lambda}) \in \R^{3n}$ of $h$ given in \eqref{eq:h_vector_field}
with $\bs{\lambda}=(\bs{\lambda}^{[0]},\bs{\lambda}^{[1]},\bs{\lambda}^{[2]}) \in \R^{3m}$ and $\bs{u}=(\bs{u}^{[0]},\bs{u}^{[1]},\bs{u}^{[2]}) \in \R^{3n}$ is defined by 
\begin{equation}\label{e:defboldh}
\bs{h} (\bs{u},\bs{\lambda}) \bydef \begin{bmatrix}
h(\bs{u}^{[0]},\bs{\lambda}^{[0]}) \\
D_\lambda h (\bs{u}^{[0]},\bs{\lambda}^{[0]}) \bs{\lambda}^{[1]}
+
D_u h (\bs{u}^{[0]},\bs{\lambda}^{[0]}) \bs{u}^{[1]}\\
\bs{h}^{[2]} (\bs{u},\bs{\lambda})
\end{bmatrix},
\end{equation}
where
\begin{alignat*}{1}
\bs{h}^{[2]} (\bs{u},\bs{\lambda}) & \bydef
D^2_\lambda h (\bs{u}^{[0]},\bs{\lambda}^{[0]}) [\bs{\lambda}^{[1]},\bs{\lambda}^{[1]}] + 
D_\lambda h (\bs{u}^{[0]},\bs{\lambda}^{[0]}) \bs{\lambda}^{[2]} \\
& \quad\qquad+
D^2_u h (\bs{u}^{[0]},\bs{\lambda}^{[0]}) [\bs{u}^{[1]},\bs{u}^{[1]}]
+
D_u h (\bs{u}^{[0]},\bs{\lambda}^{[0]}) \bs{u}^{[2]} \\
& \quad\qquad\qquad
+
2D_\lambda D_u  h (\bs{u}^{[0]},\bs{\lambda}^{[0]}) [\bs{\lambda}^{[1]},\bs{u}^{[1]}] .
\end{alignat*}
Similarly, the algebraic equations $g(\lambda)=0$ given in \eqref{eq:h_vector_field} are extended to
\[
  \bs{g}(\bs{\lambda}) = 
  \begin{bmatrix}
  g({\bs{\lambda}}^{[0]})	\\
  D_\lambda g(\bs{\lambda}^{[0]}) \bs{\lambda}^{[1]} \\
  D_\lambda g(\bs{\lambda}^{[0]}) \bs{\lambda}^{[2]} +
  D^2_\lambda g(\bs{\lambda}^{[0]}) [\bs{\lambda}^{[1]},\bs{\lambda}^{[1]}] \\  
  \end{bmatrix} .
\]

For any phase condition of the form 
$\G^{\pp_0,\pp_1}_{\yy_0,\yy_1}(x,s)$, see~\eqref{e:generalG},
the three extended equations are
\begin{alignat*}{1}
%   \G^{\bs{\pp_0},\bs{\pp_1}}_{\bs{\yy_0},\bs{\yy_1}}(\bs{x},s) 
%\begin{bmatrix}
  \G^{(\pp_0,0,0),(\pp_1,0,0)}_{\yy_0,\yy_1}
  				(\bs{x},s)	&=0,	\\[1mm]
  \G^{(\pp_1-\pp_0,\pp_0,0),(\pp_1-\pp_0,\pp_1,0)}_{\yy_1-\yy_0,\yy_1-\yy_0}
  				(\bs{x},s) &= 0,\\[1mm]
  \G^{(0,2(\pp_1-\pp_0),\pp_0),(0,2(\pp_1-\pp_0),\pp_1)}_{0,0}
  				(\bs{x},s)&= 0.
%  \end{bmatrix} .
\end{alignat*}
We note that each of these is of the form
$\G^{\bs{\pp}_0,\bs{\pp}_1}_{\tilde{\yy}_0,\tilde{\yy}_1}(\bs{x},s)$ with $\bs{\pp}_\s \in
\bs{X}$ for $\s\in\{0,1\}$ such that $\bs{\pp}_\s^*=\bs{\pp}_\s$ and $\tilde{\yy}_{\s}
\in \R$, as described in Remark~\ref{r:dependence}. Hence each of
these is conjugate equivariant.
% \[
%  \G^{\bs{\pp_0},\bs{\pp_1}}_{\yy'_0,\yy'_1}}(\bs{x}^*,s) =
%  \G^{\bs{\pp_0},\bs{\pp_1}}_{\yy'_0,\yy'_1}}(\bs{x},s)^*.
% \]

We collect all algebraic and extended phase condition equations in $\bs{G}_s(\bs{x})$.
Since taking the Fourier transform and taking the derivative with respect to~$s$ commute, the system
\[
  \bs{H}_{\!s}(\bs{x}) \bydef\begin{bmatrix} \bs{G}_s(\bs{x})\\ \bs{F}(\bs{x})\end{bmatrix} = 0 
\]
is equivalent to~\eqref{e:three}. Hence we may indeed apply the continuation technique (and code) from~\cite{QJBcontinuation} as outlined in Section~\ref{s:setup}. Denote by $\bigl( \bs{\sol{x}}^{[0]}(s),\bs{\sol{x}}^{[1]}(s) ,\bs{\sol{x}}^{[2]}(s)  \bigr)$ the resulting solution curve for $s \in[0,1]$.

%\corrc New remark. JB <<>>

\begin{remark}\label{r:triangularA}
When applying the fixed point construction of Section~\ref{s:setup}, we need to choose an approximate inverse $\bs{A}_s$ of the Jacobian. 
Since $D_{\bs{x}} \bs{H}_{\!s}(\bs{x})$ is block lower triangular with respect to the splitting $\bs{x}=(\bs{x}^{[0]},\bs{x}^{[1]},\bs{x}^{[2]})$, we will always select $\bs{A}_s$ to be block lower triangular as well. This implies that if $\bs{T}_s$ is a contraction on $\bs{X} = X^3$, then the restriction $\bs{T}^{[0]}_s(\bs{x}^{[0]})$ of $\bs{T}_s(\bs{x})$ to the first of the three components is well-defined and a contraction on $X$. Analogously, the map $(\bs{T}^{[0]}_s(\bs{x}^{[0]}),\bs{T}^{[1]}_s(\bs{x}^{[0]},\bs{x}^{[1]}))$ is a contraction on $X^2$.
\end{remark}

\begin{remark}\label{r:triangular}
Since $D_{\bs{x}} \bs{H}_{\!s}(\bs{x})$ and $\bs{A}_s$ are block lower triangular with respect to the splitting $\bs{x}=(\bs{x}^{[0]},\bs{x}^{[1]},\bs{x}^{[2]})$,
when we apply the construction of Theorem~\ref{thm:radii_bound}, injectivity of the Jacobian (see Remark~\ref{r:injectivity}) implies injectivity of 
$D_x H_s(\bs{\sol{x}}^{[0]}(s))$.
\end{remark}

\begin{lemma}\label{l:extendedradpoly}
Let the conditions of Theorem~\ref{thm:radii_bound} be satisfied for the
extended system $\bs{H}_{\!s}$  and assume that the corresponding smoothness condition~\eqref{e:smoothness} holds.
Then $\sol{x}(s) = \sol{\bs{x}}^{[0]}(s)$ solves
$H_s(x)=0$, while $\sol{\bs{x}}^{[1]}(s) = \sol{x}'(s)$ and
$\sol{\bs{x}}^{[2]}(s) = \sol{x}''(s)$.
\end{lemma}
\begin{proof}
This follows from the equivalence of $\bs{H}_s(\bs{x})$ and~\eqref{e:three}, as well as Remarks~\ref{r:uniquederivatives} and~\ref{r:triangular}.
\end{proof}

We now assume we have successfully applied Theorem~\ref{thm:radii_bound} to~$\bs{H}_s$. 
In particular, let $\hat{\bs{x}}_0 = (\hat{\bs{v}}_0,\hat{\bs{\lambda}}_0)$ and $\hat{\bs{x}}_1= (\hat{\bs{v}}_1,\hat{\bs{\lambda}}_1)$ be the end points of a line segment for which we have found, through Theorem~\ref{thm:radii_bound}, a solution curve 
\begin{equation}\label{e:insideC}
  \sol{\bs{x}}(s)=(\sol{\bs{v}},\sol{\bs{\lambda}})(s) \in \mathcal{C}_{\hr}
  \qquad \text{for } s\in [0,1]. 
\end{equation}
We also assume that the corresponding smoothness condition~\eqref{e:smoothness} holds, so that we may apply Lemma~\ref{l:extendedradpoly}.
Then we can use the following result to verify that a nondegenerate saddle-node bifurcation occurs.
\begin{proposition}\label{prop:fold}
Let $1\leq j \leq m$. Assume
\begin{subequations}
\label{e:signcondsadnod}
\begin{alignat}{1}
  \big(\hat{\bs{\lambda}}^{[1]}_0\bigr)_j + \hr &< 0 ,  \label{e:endpoint0}\\
  \bigl(\hat{\bs{\lambda}}^{[1]}_1\bigr)_j - \hr &> 0 , \label{e:endpoint1}\\
  \min\Bigl\{\bigl(\hat{\bs{\lambda}}^{[2]}_0\bigr)_j,
	  		 \bigl(\hat{\bs{\lambda}}^{[2]}_1\bigr)_j\Bigr\} -  \hr &> 0 .
			 \label{e:secondpositive}
\end{alignat}
\end{subequations}
Then the solutions curve $\sol{\bs{x}}(s)$ undergoes a unique nondegenerate fold bifurcation (folding to the right) with respect to $\lambda_j$ in the interval $s\in [0,1]$.
\end{proposition}
\begin{proof}
Let $\sol{\lambda}_j(s)=\sol{\bs{\lambda}}^{[0]}_j(s)$.
For $i=0,1,2$, denote $\hat{\bs{\lambda}}^{[i]}_s = (1-s) \hat{\bs{\lambda}}^{[i]}_0 + s \hat{\bs{\lambda}}^{[i]}_1$.
For $i=0,1,2$, let $\sol{\lambda}^{(i)}_j(s)$ denote the $i$-th derivative 
of $\sol{\lambda}_j(s)$.
It follows from Lemma~\ref{l:extendedradpoly} and Equation~\eqref{e:insideC}
that 
\begin{equation}\label{e:hrclose}
  |\sol{\lambda}^{(i)}_j(s)- (\hat{\bs{\lambda}}^{[i]}_s)_j |=
  |\sol{\bs{\lambda}}^{[i]}_j(s)- (\hat{\bs{\lambda}}^{[i]}_s)_j | \leq \hr
  \qquad \text{for } i=0,1,2.
\end{equation}
It then follows from~\eqref{e:endpoint0}--\eqref{e:endpoint1} that $\sol{\lambda}'_j(0)<0$
while $\sol{\lambda}'_j(1)>0$, hence by the intermediate value theorem 
there exists an $\sstar \in (0,1)$ such that $\sol{\lambda}'_j(\sstar)=0$. Furthermore, since 
\[
        \bigl(\hat{\bs{\lambda}}^{[2]}_s\bigr)_j = 
 (1-s)  \bigl(\hat{\bs{\lambda}}^{[2]}_0\bigr)_j +
    s   \bigl(\hat{\bs{\lambda}}^{[2]}_1\bigr)_j ,
\]
it follows from~\eqref{e:hrclose} and~\eqref{e:secondpositive}
that $\sol{\lambda}''_j(s) >0$ for all $s \in [0,1]$.
In particular, $\sol{\lambda}''_j(\sstar) >0$ and
$\sol{\lambda}_j(s) \geq \sol{\lambda}_j(\sstar)$ for all $s\in[0,1]$,
and besides $s=\sstar$ there is no other zero of $\sol{\lambda}'_j(s)$ on $[0,1]$. 
\end{proof}

Alternative conditions which lead to the same result, but with the curve folding to the left, are
\begin{subequations}
\label{e:altsigncondsadnod}
\begin{alignat}{1}
  \big(\hat{\bs{\lambda}}^{[1]}_0\bigr)_j - \hr &> 0 , \\
  \bigl(\hat{\bs{\lambda}}^{[1]}_1\bigr)_j + \hr &< 0 , \\
  \max\Bigl\{\bigl(\hat{\bs{\lambda}}^{[2]}_0\bigr)_j,
	  		 \bigl(\hat{\bs{\lambda}}^{[2]}_1\bigr)_j\Bigr\} +  \hr & < 0 .
\end{alignat}
\end{subequations}

%\corrc New remark. Elena: does this capture what we should say about the rescaling? It does not go into implementation, but does flag the need for rescaling.  JB Seems good to me. E<<>>
\begin{remark}\label{r:rescalingderivatives}
Since the stepsize may be quite (or even very) small in practice, the first derivative $\bs{\sol{x}}^{[1]}(s)=(\bs{\sol{x}}^{[0]})'(s)$ can be of an entirely different size than $\bs{\sol{x}}^{[0]}(s)$, and for the second derivative this holds \emph{a fortiori}. In such a situation, the variable components $\bs{x}^{[0]}$, $\bs{x}^{[1]}$ and $\bs{x}^{[2]}$ are not of commensurable magnitude, and a uniform norm, as in~\eqref{e:maxnorm} is not appropriate, and indeed using it unaltered makes the conditions~\eqref{e:signcondsadnod} unachievable. This obstacle is overcome by rescaling the variable components $\bs{x}^{[1]}$ and $\bs{x}^{[2]}$  appropriately, with scale parameters that are linear and quadratic in the stepsize, respectively.

Furthermore, since inequalities~\eqref{e:endpoint0}--\eqref{e:endpoint1} are evaluated in the endpoints only, we may apply a noncontinuation version of Theorem~\ref{thm:radii_bound} to each endpoint separately (for the smaller extended system discussed in Remark~\ref{r:precheck_saddle}), which helps in verifying~\eqref{e:endpoint0}--\eqref{e:endpoint1}.
\end{remark}

%\corrc Another new remark. JB <<>>
% \begin{remark}\label{r:unique}
% Regarding uniqueness, in view of Remark~\ref{r:triangularA} and inequality~\eqref{e:smoothness}, applying Theorem~\ref{thm:radii_bound} to $\bs{H}_s$
% identifies a unique, smooth solution curve $\{\sol{\bs{x}}^{[0]}(s)\}_{s\in[0,1]}$ in a cylinder $\widetilde{\bs{\mathcal{C}}}^{[0]}_{\hr} \subset X$ of radius~$\hr$ around the line segment~$\{ \hat{\bs{x}}^{[0]}_s\}_{s\in[0,1]}$,
% see Remark~\ref{r:cylinder}.
% Proposition~\ref{prop:fold} provides conditions which guarantee that a unique nondegenerate fold occurs along this solution curve, and that  with respect to $\lambda_j$ there are no other folds, degenerate or not, along $\{\sol{\bs{x}}^{[0]}(s)\}_{s\in[0,1]}$.
% \end{remark}

\begin{remark}\label{r:precheck_saddle}
In a similar, but easier, fashion one may verify that \emph{no} bifurcation occurs with respect to~$\lambda_{j'}$ by checking that
\[
  \min\Bigl\{\bigl(\hat{\bs{\lambda}}^{[1]}_0\bigr)_{j'},
	  		 \bigl(\hat{\bs{\lambda}}^{[1]}_1\bigr)_{j'}\Bigr\} -  \hr > 0 
\quad\text{or}\quad
  \max\Bigl\{\bigl(\hat{\bs{\lambda}}^{[1]}_0\bigr)_{j'},
	  		 \bigl(\hat{\bs{\lambda}}^{[1]}_1\bigr)_{j'}\Bigr\} +  \hr < 0 .
\]
Naturally, since this does not involve the second derivative, if one merely wants to exclude bifurcations it suffices to apply Theorem~\ref{thm:radii_bound} to the smaller extended system for $(\bs{x}^{[0]},\bs{x}^{[1]})$. 
\end{remark}

%\corrc The archlength parameter can be reascaled and a rescaling of $s$ entails a rescaling of the first and second derivative. We therefore want to consider $s\in [0, S]$ with an appropriately chosen $S$. In particular, $S$ is chosen such that the order of magnitude of $\|x'_0-x'_1\| \approx \|x_0-x_1\|$, that is $S = 1/h$, with $h$ the step size of the continuation.  E<<>>

%\corrc Yes, I have tried to ``formalize'' this in Remark~\ref{r:rescalingderivatives}.  JB <<>>

%\corrc Seems good and enough for me. E<<>>

\section{Eigenvalue considerations}
\label{s:eigenvalues}
%!TEX root = bifcont.tex

%\corrc Still need to reorder variables in other sections. JB <<>>
%\corrc Still need to resolve notation clashes (especially with $w$ in the introduction) JB <<>>

Traditionally, saddle-node bifurcations are identified in terms of a simple eigenvalue crossing $0$. Here we discuss how our nondegenerate folds, as described in Definition~\ref{defn:saddle-node} by the local parabolicity in~\eqref{e:nondegsadnod} of the solution curve, relates to such eigenvalue considerations.
We consider the case of a fold bifurcation for~\eqref{e:prelim} with respect to $\lambda_2 = \mu$.
In particular, $\lambda=(\tau,\mu) \in \R^2$ and
$h:\R^n \times \R^{2} \to \R^{n}$ is given by $\tau f(u,\mu)$,
where $f$ is the vector field in \eqref{e:initial} and $\tau$ is the normalized period.

We start by describing the information on the eigenvalues of 
$D_{x} H_s(\sol{x}(s))$,
and subsequently relate this to stability information for the periodic orbits of~\eqref{e:initial}.
We collect the variables $v$ and $\tau$ in $w=(v,\tau) \in  \Y=(\ell^1_\nu)^n \times \C$. 
Let
\[
  \tH (x) \bydef \begin{bmatrix} F(x) \\ G^\phase(x) \end{bmatrix} = 0 .
\]
We then write 
\begin{equation}\label{e:blockDxHs}
  D_x H_s(\sol{x}(s)) = \left [
  \begin{array}{cc}
   D_w \tH(\sol{x}(s)) & D_\mu \tH(\sol{x}(s)) \\
   q^\cont_w(s) & q^\cont_\mu(s)
  \end{array}
  \right] ,
\end{equation}
where we unravel the notation for the continuation equation~\eqref{e:continuationequation} through the use of $q^\cont_s = (q^\cont_w, q^\cont_\mu)(s) =  (q^\cont_v, q^\cont_\tau,q^\cont_\mu)(s)$.
We infer from Remark~\ref{r:injectivity}
that $D_x H_s(\sol{x}(s))$ is invertible as a map from $X$ to
$  X_1 = (\ell^1_{\nu,1})^n \times \C^m$,
where
\[ \ell^1_{\nu,1} \bydef \{ \tilde{v} \in \ell^1_{\nu} \,:\,  \K \tilde{v} \in \ell^1_{\nu} \} ,
\]
which is a Banach space when equipped with the norm $\| \tilde{v} \|_{\nu,1} = \sum_{k\in\Z} |\tilde{v}_k| \nu^{|k|} (|k|+1)$.

Our central object of interest is
\[
  \Q_s \bydef D_w \tH(\sol{x}(s)),
\]
which is conjugation invariant, i.e., $\Q_s w^* = (\Q_s w)^*$, since $\tH$ is and $\sol{x}(s) \in \XS$.
Furthermore, $\Q_s$ is a bounded operator from~$\Y$ to $\Y_1 \bydef (\ell^1_{\nu,1})^n \times \C$,
but it may also be interpreted as an unbounded operator on 
 $\Y$. 
Finally, $\Q_s$ is Fredholm of index $0$. 
%\corrc Should we say anything more about Fredholm in this paper at all? JB <<>>

We note that, by differentiating the identity $\tH(\sol{x}(s))=0$ with respect to $s$, we have 
\begin{equation}\label{e:easyQslinear}
\Q_s \sol{w}'(s) + D_\mu \tH(\sol{x}(s)) \sol{\mu}'(s)=0.
\end{equation}
Hence at $s=\sstar$, where $\sol{\mu}'(\sstar)=0$, we have that $\Q_{\sstar}$, and thus $D_w \tH(\sol{x}(\sstar))$, has $\sol{w}'(\sstar)$ as an eigenvector associated to the zero eigenvalue.
Indeed $\sol{w}'(\sstar)$ is not trivial in view of smoothness of the solution curve (Inequality~\eqref{e:smoothness}).

The next remark guarantees that $\Q_s$ does not have a zero eigenvalue for $s\neq \sstar$.
\begin{remark}\label{r:kernelimpliesvertical}
If $\Q_s$ has an eigenvector $w_0$ associated to the zero eigenvalue for some $s \in [0,1]$,
then it follows from invertibility of $D_x H_s(\sol{x}(s))$ that $\langle  q^\cont_w(s) ,w_0 \rangle \neq 0$. In turn it follows that $W_0 \bydef (c_0 w_0,0)$ with 
\[
  c_0 \bydef - \frac {D_s G^\cont_s(\sol{x}(s)) } { \langle q^\cont_w(s)
    ,w_0 \rangle} ,
\] 
solves
$D_x H_s(\sol{x}(s)) W_0 = -D_s H_s(\sol{x}(s))$. Hence, by invertibility of $D_x H_s(\sol{x}(s))$, we find that $\sol{x}'(s)=(c_0 w_0,0)$,
because we have the identity
\[
  D_x H_s(\sol{x}(s)) \sol{x}'(s)  + D_s H_s(\sol{x}(s)) =0
  \qquad\text{for all } s \in [0,1].
\]
In particular, this implies that $\sol{\mu}'(s)=0$, so that we conclude from the uniqueness statement in Proposition~\ref{prop:fold}
that $\Q_s$ has eigenvalue zero at $s=\sstar$ only.
\end{remark}

Next we argue that the geometric multiplicity of the zero eigenvalue of $\Q_{\sstar}$ is $1$. Namely, if the $0$-eigenspace of $\Q_{\sstar}$ is two (or higher) dimensional, then it is straightforward to construct a nontrivial element in the kernel of $D_x H_s(\sol{x}(s))$, which contradicts its injectivity, which was established in Remark~\ref{r:injectivity}.
It will take a bit more work (see below) to analyze the \emph{algebraic} multiplicity of the zero eigenvalue, which may be higher than $1$.

We introduce the dual space $\Y'=(\ell^\infty_{\nu^{-1}})^n \times \C$,
where
$\ell^\infty_{\nu^{-1}} \bydef
\{ \tilde{v} \in \C^\Z \,:\, \sup_{k\in \Z} |\tilde{v}_k| \nu^{-|k|} <\infty\}$,
and let $\checkwstar \in \Y'$ be an eigenvector of the transpose $\Q'_{\sstar}$ associated to the zero eigenvalue.
Here the transpose has the usual definition:
$\langle \Q'_{\sstar} \check{w} , w \rangle = \langle  \check{w} , \Q_{\sstar} w \rangle $ for all $w\in \Y$, $\check{w}\in \Y'$,
where the dual pairing uses the slightly abused notation, cf~\eqref{e:bilinear}, 
\[
\langle \check{w} , w \rangle  =
\sum_{i=1}^{n} \sum_{k\in \Z} (\check{v}_i)_k (v_i)_ k 
	 								+ \check{\tau} \;\! \tau   .
\]
Since we can restrict $\Q_{\sstar}$ to the conjugate symmetric set $\{ w \in \Y : w^*=w\}$, and the eigenvector $\sol{w}'(\sstar)$ lies in this set, we may also assume that $\checkwstar^*=\checkwstar$.
Then the range of $\Q_{\sstar}$ can be characterized as
\begin{equation}\label{e:rangeQsstar}
  \range \Q_{\sstar} = \bigl\{ w \in \Y_1 : \langle \checkwstar , w \rangle  =0 \bigr\} .
\end{equation}

We now use a standard trick to obtain information about eigenvalues.
Let $0 \neq \checkwstar$ be such that $\Q'_{\sstar} \checkwstar =0$. 
Invertibility of $D_x H_s(\sol{x}(s))$ implies that
its transpose is invertible as a linear map from $X_1'$ to $X'$.
Hence we see from~\eqref{e:blockDxHs} that
\begin{equation}\label{e:nontriv1}
  \langle \checkwstar , D_\mu \tH(\sol{x}(\sstar)) \rangle  \neq 0.
\end{equation}
% A consequence of~\eqref{e:nontriv1} is that $D_\mu \tH(\sol{x}(s))$ is nontrivial.
% Hence it follows from~\eqref{e:easyQslinear} that whenever $\Q_s$ has eigenvalue zero, then $\sol{\mu}'(s)=0$,
% so that we conclude from the uniqueness statement in Lemma~\ref{l:fold}
% that $\Q_s$ has eigenvalue zero at $s=\sstar$ only.
The second derivative $\sol{x}''(s)=(\sol{w}''(s),\sol{\mu}''(s))$ satisfies
\begin{alignat*}{1}
   D_{w} \tH(\sol{x}(s)) \sol{w}''(s)
   + D_{\mu} \tH(\sol{x}(s)) \sol{\mu}''(s)   &=  \\
   &\hspace*{-4cm}
   - D^2_{w,w} \tH(\sol{x}(s))[\sol{w}'(s),\sol{w}'(s)]  
  - 2 D^2_{w,\mu} \tH(\sol{x}(s))[\sol{w}'(s),\sol{\mu}'(s)] 
  - D^2_{\mu,\mu} \tH(\sol{x}(s))[\sol{\mu}'(s),\sol{\mu}'(s)] .
\end{alignat*}
When we evaluate this at $s=\sstar$ and apply $\checkwstar$ to the result, we obtain
\begin{equation}\label{e:identity2}
    \langle \checkwstar , D_\mu \tH(\sol{x}(\sstar))  \rangle \sol{\mu}''(\sstar) = -    \langle \checkwstar , D^2_{w,w} \tH(\sol{x}(\sstar))[\sol{w}'(\sstar),\sol{w}'(\sstar)] \rangle .
\end{equation}
Since  $\mu''(\sstar) \neq 0$ by nondegeneracy of the fold (see~\eqref{e:nondegsadnod} and Proposition~\ref{prop:fold}), we conclude from~\eqref{e:nontriv1} and~\eqref{e:identity2}  that   
\begin{equation}\label{e:nontriv2}
   \langle \checkwstar , 
   D^2_{w,w} \tH(\sol{x}(\sstar))[\sol{w}'(\sstar),\sol{w}'(\sstar)]
   \rangle   \neq 0.
\end{equation}	
We note that all of the elements $\checkwstar$, $D_\mu \tH(\sol{x}(\sstar)) $
and $D^2_{w,w} H_s(\sol{x}(\sstar))[\sol{w}'(\sstar),\sol{w}'(\sstar)] $
are conjugate symmetric, which is consistent with $\sol{\mu}''(\sstar)$ being real.
	
We are now ready to analyze the situation associated to higher algebraic multiplicity of the zero eigenvalue.
The eigenvalue problem for $\Q_s$ is 
\[
  \QQ_\alpha(w,s) \bydef \begin{bmatrix}
   \Q_s w - \alpha w \\  \langle q^\cont_w(\sstar) , w \rangle - \langle q^\cont_w(\sstar) ,\sol{w}'(\sstar) \rangle 
   \end{bmatrix} =0,
\]
which has a zero $\QQ_0(\sol{w}'(\sstar),\sstar)=0$.
The derivative at this zero is
\[
  D \QQ_0 
 = \begin{bmatrix}
   \Q_{\sstar} 
   & D^2_{w,w} \tH(\sol{x}(\sstar)) [\sol{w}'(\sstar),\sol{w}'(\sstar)]  \\
   q^\cont_w(\sstar) & 0 
  \end{bmatrix} .
\]
Since $D_x H_{\sstar}(\sol{x}(\sstar))$ is invertible
we infer that, see~\eqref{e:blockDxHs}, the range of 
\[ 
 \begin{bmatrix}
   \Q_{\sstar}\\
     q^\cont_w(\sstar) 
  \end{bmatrix}
\]
has co-dimension $1$ in $X_1$.
It then follows from~\eqref{e:rangeQsstar} and \eqref{e:nontriv2} that
$D \QQ_0$ is invertible.
Hence we conclude from the implicit function theorem that  $\QQ_\alpha(w,s)=0$ has a locally unique smooth branch of solutions $\QQ_\alpha(\bar{w}(\alpha),\bar{s}(\alpha))=0$.

The local dependence of $\bar{s}$ on $\alpha$ is governed by the algebraic multiplicity of the zero eigenvalue as follows.
By differentiating $D_w \tH(\sol{x}(\bar{s}(\alpha))) \bar{w}(\alpha) - \alpha \bar{w}(\alpha) =0$ with respect to $\alpha$ we obtain
\begin{alignat}{1}
  &D^2_{w,w} \tH(\sol{x}(\bar{s}(\alpha))) [\sol{w}'(\bar{s}(\alpha)),\bar{w}(\alpha)]  \bar{s}'(\alpha)+ 
  D^2_{\mu,w} \tH(\sol{x}(\bar{s}(\alpha))) [\sol{\mu}'(\bar{s}(\alpha)),\bar{w}(\alpha)]
  \bar{s}'(\alpha) \nonumber \\
 & \hspace*{7.5cm}+ D_w \tH(\sol{x}(\bar{s}(\alpha))) \bar{w}'(\alpha) - \alpha \bar{w}'(\alpha) -\bar{w}(\alpha)=0.
  \label{e:differentiatealpha}
\end{alignat}
Substituting $\alpha=0$ we obtain (recalling that $\Q_{\sstar}=D_w \tH(\sol{x}(\sstar)$)
\begin{equation}\label{e:secondfirstalpha}
 D^2_{w,w} \tH(\sol{x}(\sstar)) [\sol{w}'(\sstar),\sol{w}'(\sstar)]     
 \bar{s}'(0)  + \Q_{\sstar} \bar{w}'(0) = \bar{w}(0),
\end{equation}
and applying $\checkwstar$ to this  result leads to
\begin{equation}\label{e:secondonlyalpha}
  \bigl\langle \checkwstar , D^2_{w,w} \tH(\sol{x}(\sstar)) [\sol{w}'(\sstar),\sol{w}'(\sstar)]  \bigr\rangle  \, \bar{s}'(0) 
  =  \langle \checkwstar , \sol{w}'(\sstar) \rangle.
\end{equation}
We see from~\eqref{e:nontriv2} that $\bar{s}'(0) \neq 0$ if and only if $\langle \checkwstar , \sol{w}'(\sstar) \rangle \neq 0$. In view of~\eqref{e:rangeQsstar} 
the latter is equivalent to $\sol{w}'(\sstar)$ not being in the range of $\Q_{\sstar}$,
i.e., the algebraic multiplicity of the zero eigenvalue being 1. The contrapositive is that $\bar{s}'(0)$ vanishes if the algebraic multiplicity is larger than 1. We consider each case below.

In the former (algebraic multiplicity 1) case we see from~\eqref{e:secondonlyalpha}, combined with conjugate symmetry of  
$\checkwstar$, that $D^2_{w,w} \tH(\sol{x}(\sstar)) [\sol{w}'(\sstar),\sol{w}'(\sstar)]$
and $\sol{w}'(\sstar)$, 
 that $\bar{s}'(0) \in \R \setminus \{0\}$.
By inverting the relation, we find that $\alpha = (\bar{s}'(0))^{-1} (s-\sstar) + O((s-\sstar)^2)$, hence the eigenvalue crosses zero with nonzero speed as  we go through the fold and we have a classical saddle-node bifurcation.

In the latter (algebraic multiplicity larger than 1) case we differentiate~\eqref{e:differentiatealpha} again, and substitute $\alpha=0$. Using the information that $\bar{s}'(0)=0$ and
\begin{equation}\label{e:generalized eigenvector}
  D_w \tH(\sol{x}(\sstar)) \bar{w}'(0) = \bar{w}(0),
\end{equation}
 which follows from~\eqref{e:secondfirstalpha}, we obtain
\[
 D^2_{w,w} \tH(\sol{x}(\sstar)) [\sol{w}'(\sstar),\sol{w}'(\sstar)]     
 \bar{s}''(0)  + \Q_{\sstar} \bar{w}''(0) = 2 \bar{w}'(0) .
\]
Once again applying $\checkwstar$ to this, we find
\[
  \bigl\langle \checkwstar , D^2_{w,w} \tH(\sol{x}(\sstar)) [\sol{w}'(\sstar),\sol{w}'(\sstar)]  \bigr\rangle \bar{s}''(0) 
  =  2 \langle \checkwstar , \bar{w}'(0) \rangle.
\]
There are two possible scenarios. Either $\bar{s}''(0) \neq 0$ and $\bar{w}'(0)$ is not in range of $\Q_{\sstar}$ and thus, in view of~\eqref{e:generalized eigenvector},
the algebraic multiplicity of the zero eigenvalue is $2$. 
Or $\bar{s}''(0) =0$ and 
$ \Q_{\sstar} \bar{w}''(0) = 2 \bar{w}'(0)$ and 
the algebraic multiplicity of the zero eigenvalue is larger than $2$. 
 
We can now repeat the above arguments inductively. We conclude that 
$s=\sstar + C \alpha^N + O(\alpha^{N+1})$
for some $C \in \R \setminus\{0\}$, where $N$ is the algebraic multiplicity of the zero eigenvalue of 
$\Q_{\sstar}=D_w \tH(\sol{x}(\sstar))$ at the fold (which always has geometric multiplicity 1).
Inverting the relation, we see that 
$\alpha = C^{1/N} (s-\sstar)^{1/N} + O((s-\sstar)^{(N+1)/N})$, which gives detailed information on the dynamics of the $N$ eigenvalues of $\Q_s=D_w \tH(\sol{x}(s))$ that coalesce at $0$ when we move through the fold. In particular, the number of negative eigenvalues changes by $1$ (not counting complex conjugate pairs).

Before we discuss the implications of this information for the eigenvalue problem of $D_v F(\sol{x}(s))$, we exclude some marginal behaviour. More precisely, in the next two remarks, we exclude two possible ``trivialities'' of the fold.
\begin{remark}
First, the solution at the fold cannot be an equilibrium solution.
By contradiction, assume that $\sol{v}_k(\sstar)=0$ for all $k \neq 0$,
and $f(\sol{v}_0(\sstar),\sol{\mu}(\sstar))=0$.
We write
\begin{equation}\label{e:blockDwtH}
  D_w \tH(\sol{x}(\sstar)) = \begin{bmatrix}
   D_v F (\sol{x}(\sstar)) & D_\tau F (\sol{x}(\sstar)) \\
   q^\phase_v & 0
  \end{bmatrix} .
\end{equation}
Since the kernel of $D_w \tH(\sol{x}(\sstar))$ is one dimensional,
and $D_\tau F (\sol{x}(\sstar)) = f(\sol{v}_0(\sstar),\sol{\mu}(\sstar))$ vanishes, we conclude that that $(v,\tau)=(0,1)$ spans the kernel.
By differentiating the identity $\tH(\sol{x}(s))=0$  
with respect to $s$, evaluating at $s=\sstar$ and using
that $\sol{\mu}'(\sstar)=0$, we then conclude that 
$\sol{v}'(\sstar)=0$ and $\sol{\tau}'(\sstar) \neq 0$.
In turn this implies that 
\[
 D^2_{w,w} \tH(\sol{x}(\sstar))[\sol{w}'(\sstar),\sol{w}'(\sstar)]
 = D^2_{\tau,\tau} \tH(\sol{x}(\sstar))
[\sol{\tau}'(\sstar),\sol{\tau}'(\sstar)],
 \]
which vanishes since $\tH$ is linear in $\tau$. This contradicts~\eqref{e:nontriv2}.
\end{remark}

% where $\hat{q} = q^\phase_v$, \red{see~\eqref{e:choiceofphaseq}, notation inconsistent!}.
% The main part $D_v F (\sol{x}(\sstar))$ would be of the form
% \[
%   (D_v F (\sol{x}(\sstar)) \tilde{v})_{k} =
%   \Bigl( -\imag k I_n
%   + \sol{\tau}(\sstar) D_u f(\sol{v}_0(\sstar),\sol{\mu}(\sstar)) \Bigr)
%    \tilde{v}_k ,
% \]
% i.e., diagonal in $k$, with $v_k \in \C^n$ and $I_n$ the identity matrix on $\C^n$.
% Furthermore, $D_\tau F (\sol{x}(\sstar))=\tau f(\sol{v}_0(\sstar),\sol{\mu}(\sstar)$ vanishes, and  $\red{\hat{q}_0}=(q^\cont_v)_0(s)=0$ by~\eqref{e:choiceofphaseq}.
% It follows from the fact that the kernel of $D_w \tH(\sol{x}(\sstar))$ is one dimensional that $D_u f(\sol{v}_0(\sstar),\sol{\mu}(\sstar))$ must be invertible. Moreover, we find that this kernel is spanned by $ $
% And the (unique up to scalar) solution is $\sol{\tau}'(\sstar)=1$, $\sol{v}'(\sstar)=0$.

\begin{remark}\label{r:notphaseshift}
The second scenario that we would like to exclude is that the fold inadvertently occurs due to a shift in time. In particular, we want to confirm that
\[
  \bigl\langle q^\phase_v  , \imag \K \sol{v}(\sstar)\bigr\rangle \neq 0. 
\]
This is achieved through a computational check, cf.~the smoothness condition~\eqref{e:smoothness}:
\[
 \hat{r}  \, \langle q^\phase , b \rangle 
\neq
 \langle  q^\phase , \imag \K [(1-s)\hat{x}_0 +s \hat{x}_1]
 \rangle  
\quad \text{ for any } s \in [0,1] \text{ and } b\in B_1(0),
\]
which in practice is satisfied since 
$\overline{\imag \K {\hx}_0} \approx q^\phase
\approx\overline{\imag \K {\hx}_1}$ and $\hr \ll 1$.
\end{remark}

We now relate eigenvalues of $\Q_s$ to eigenvalues of $D_v F(\sol{x}(s))$,
which in turn correspond to characteristic multipliers of
the periodic solution
\[
  \sol{u}(t) = \sum_{k\in \Z} \sol{v}_k e^{\imag k t}.  
\]
Namely, an eigenvalue-eigenvector pair $(\tilde{\alpha},\tilde{a})$ of $D_v F(\sol{x}(s))$ solves
\[
   \imag \K \tilde{a} - \sol{\tau}(s)  D_v \widehat{f} (\sol{v}(s),\sol{\mu}(s)) \tilde{a}  = \tilde{\alpha} \tilde{a}.
\]
Hence $a(t)=\sum_{k\in \Z} \tilde{a}_k e^{\imag k t}$ is a $2\pi$-periodic solution of the linearized problem
\[
  \dot{a}(t) - \sol{\tau}(s) D_u f(\sol{u}(t;s),\sol{\mu}(s)) a(t) = \tilde{\alpha} a(t),
\]
and $e^{2\pi \tilde{\alpha}}$ is thus a characteristic multiplier of the system.
In particular, the eigenvalues of $D_v F(\sol{x}(s))$ contain the information about linearized stability of the periodic orbit $\sol{u}$.
We now return to the relation between the eigenvalues of $\Q_s$ and the eigenvalues of $D_v F(\sol{x}(s))$.

Let $\tilde{w}=(\tilde{v},\tilde{\tau})$ be an eigenvector with eigenvalue $\alpha$ of $\Q_s$. For generalized eigenvectors a similar analysis goes through; we leave the details to the reader.
Since $D_v F (\sol{x}(s))$ has a zero eigenvector $V_0 \bydef \imag \tau_0^{-1} \K  \sol{v}(s)$,  with $\tau_0 \bydef \sol{\tau}(s)$, and $D_\tau F (\sol{x}(s)) = -\widehat{f}(\sol{x}(s)) = -\imag \K  \tau_0^{-1} \sol{x}(s) = - V_0$,
we infer  that
\begin{equation}\label{e:extendedeigenvalueproblem}
  \Q_s \begin{bmatrix}
   \tilde{v} \\ \tilde{\tau} 
  \end{bmatrix}
  =
    \begin{bmatrix}
   D_v F (\sol{x}(s)) & - V_0  \\
   q^\phase_v & 0
  \end{bmatrix}
 \begin{bmatrix}
   \tilde{v} \\ \tilde{\tau} 
  \end{bmatrix} =
  \alpha
  \begin{bmatrix}
   \tilde{v} \\ \tilde{\tau} 
  \end{bmatrix} .
\end{equation}
In view of Remark~\ref{r:notphaseshift} we have that $\langle q^\phase_v , V_0 \rangle \neq 0$ for $s$ near $\sstar$.

We now consider three cases: $\alpha=0$, $\alpha^2 = -\langle q^\phase_v , V_0 \rangle$,
and all other $\alpha$.

If $\alpha=0$ then it follows immediately  from~\eqref{e:extendedeigenvalueproblem} that
$D_v F (\sol{x}(s)) \tilde{v} = \tilde{\tau} V_0$, hence $\tilde{v}$ is either an eigenvector or a generalized eigenvector for eigenvalue zero of $D_v F (\sol{x}(s))$. Furthermore, $\tilde{v}$ is not a multiple of $V_0$ since $\langle q^\phase_v ,\tilde{v} \rangle = 0$.

If $\alpha \neq 0 $ and $\alpha^2 \neq -\langle q^\phase_v , V_0 \rangle$,
then we set 
$
  \tilde{V}=\tilde{v} + \frac{\tilde{\tau}}{\alpha} V_0
$
and we conclude from~\eqref{e:extendedeigenvalueproblem} that $D_v F (\sol{x}(s)) \tilde{V}= \alpha \tilde{V}$.
Additionally we find $\langle q^\phase_v , \tilde{v} \rangle =\alpha \tilde{\tau}$,
hence if $\tilde{\tau}\neq 0$ then $\langle q^\phase_v , \tilde{V} \rangle = \frac{\tilde{\tau}}{\alpha} (\alpha^2 + \langle q^\phase_v , V_0 \rangle) \neq 0 $. 
We infer that 
$\tilde{V}$ is an eigenvector of $D_v F (\sol{x}(s))$ with eigenvalue $\alpha$. 
If $\tilde{\tau}=0$ then $\tilde{V}=\tilde{v} \neq 0$ and we reach the same conclusion.

If $\alpha^2 = -\langle q^\phase_v , V_0 \rangle$, then
$(\tilde{v},\tilde{\tau})=(V_0,-\alpha)$  solves~\eqref{e:extendedeigenvalueproblem}.
There is no relation to eigenvalues of $D_v F (\sol{x}(s))$.

% \corrc Need to say (and check for ourselves! but without providing any
% details) that this analysis can be extended to generalized eigenvectors! JB
% <<>>

Finally, we note that since $\langle q^\phase_v, V_0 \rangle \neq 0$ at $s=\sstar$,
the eigenvalues $\alpha \approx C^{1/N} (s-\sstar)^{1/N}$
emanating from a fold bifurcation all correspond to eigenvalues of $D_v F (\sol{x}(\sstar))$ and hence to characteristic multipliers of the periodic orbit $\sol{u}(t)$.

\section{Hopf bifurcation}
\label{s:hopf}
%!TEX root = bifcont.tex

%\subsection{Hopf bifurcations}

A Hopf bifurcation is characterized by the junction of a family of equilibria and a family of periodic orbits.
Let $\sy(\mu)$ denote a smooth family of equilibria: $f(\sy(\mu),\mu)=0$ for $\mu \in [\mu_0,\mu_1]$.
Let $\{ \sol{x}(s) \}_{s\in[0,1]}$ be a solution curve of periodic orbits of~\eqref{e:initial}, using the notation introduced in Section~\ref{s:setup}. We write  $\sol{x}(s)=(\sol{v}(s),\sol{\tau}(s),\sol{\mu}(s))$.

In a Hopf bifurcation we have $\sol{v}(\sstar)=\sy(\mu(\sstar))$,
which is to be interpreted as meaning that the Fourier coefficients $\sol{v}(\sstar)$ correspond to the \emph{stationary} state $\sy(\mu(\sstar))$.

\begin{definition} \label{defn:Hopf}
{\em
We say that there is a \emph{nondegenerate Hopf bifurcation} (with respect to $\mu$) at $\mustar \in (\mu_1,\mu_2)$ 
if there is an $\sstar \in (0,1)$ such that 
\begin{equation*}
  \sol{\mu}(\sstar)=\mustar \qquad\text{and}\qquad \sol{v}(\sstar)=\sy(\mustar),
\end{equation*}
while $\sol{v}(s)$ is not a stationary solution for any $s \neq \sstar$, and 
\begin{equation}\label{e:nondeghopf}
  \sol{\mu}'(\sstar) = 0,
  \qquad 
  \sol{\mu}''(\sstar) \neq 0.  
\end{equation}
}
\end{definition}

In this section we explain how to establish such nondegenerate  Hopf bifurcations using a blowup (or desingularization) technique.
Note that we do not impose any eigenvalue restrictions in the description~\eqref{e:nondegsadnod} of a nondegenerate Hopf bifurcation.

\begin{remark}\label{r:hopftau}
Hopf bifurcations with respect to the normalized period $\tau$ also fit into our framework. Although they are not Hopf bifurcations in the traditional sense, these do appear naturally when studying Hamiltonian problems, or stationary states in partial differential equations with periodic boundary conditions when varying the size of the domain. In Section~\ref{s:hamiltonianexample} we present an example.
\end{remark}

% We assume that we have a smooth family of stationary points
% $(\stat{\mu},\sy) \in \R \times \R^{n}$ of~\eqref{e:initial}, that is solutions of
% %
% \[
% f(\mu,u)=0
% \]
% %
% parametrized by $\sigma \in I$. We assume $\stat{\mu}'(\sigma) \neq 0$, so that we may locally interpret the curve of equilibria as being parametrized by $\mu$.
%
%
% A Hopf bifurcation occurs at $\mu =\stat{\mu} (\sigmastar)$ if a pair of simple conjugate eigenvalues of $D_u f(\stat{\mu}(\sigma),\sy(\sigma))$ crosses the imaginary axis at nonzero speed at $\sigma=\sigmastar$. At such a bifurcation a periodic orbit of
% %
% \[
% \dot{u}=f(\mu,u)
% \]
% %
% is born (either for $\mu>\stat{\mu}(\sigmastar)$ or $\mu<\stat{\mu}(\sigmastar)$)
% provided an additional nondegeneracy condition is met (discussed in more detail below).

% \blue{Need to state here what exactly we mean by a nondegenerate Hopf bifurcation.}
% {\color{green} possible: the amplitude of the periodic orbit grows as $\sqrt{\mu}$}

As already explained in the introduction, we rescale time and put ourselves in the  context of~\eqref{e:prelim}. In effect, we set $\lambda=(\tau,\mu)$ and replace $f$ by $\tilde{f}=\tau f$. 
%and we write $g(\lambda)$ for $g(\mu)$. Clearly equilibria of $f$ and
%$\tilde{f}$ are in one-to-one correspondence.
% We note that a (Hopf) bifurcation with respect to $\lambda_1$ is not a
% traditional (Hopf) bifurcation \blue{(this remark should likely be moved to
% Section~\ref{s:saddlenode})}, although such problems may appear when studying
% partial differential equations with periodic boundary conditions when varying
% the size of the domain.
To resolve a branch of periodic orbits all the way into the Hopf bifurcation, we introduce the rescaling $u(t)=\y+a \bar{u}(t)$, with $a\in \R$, and $\y=\y(\mu)\in \R^n$ solving
\[
  f(\y,\mu)=0 .
\]
The ODE for $\bar{u}$ becomes
\begin{equation}\label{e:deftildef}
  \dot{\bar{u}} = \bar{f}(\bar{u},\bar{\lambda}) \bydef 
  \begin{cases}
  \frac{\tilde{f}(\y+a \bar{u},\lambda)-\tilde{f}(\y,\lambda) }{a}  & \text{if } a \neq 0 ,\\
  D_u \tilde{f}(\y,\lambda)\bar{u}   & \text{if } a =0,  
  \end{cases}
\end{equation}
%\corrc why this definition? Elena. <<>>
%\corrc I don't understand the question. Is there something wrong with the definition? JB <<>>
%\corrc no, I just find it less intuitive than the Taylor series introduced previously Elena. <<>>
%\corrc in the code of course just use the series in $a$, but by using~\eqref{e:deftildef} we do not need to introduce notation for such a series. JB. <<>>
where $\bar{\lambda}=(\tau,a,\y,\mu)$.
We note that $\bar{f}$ is again polynomial in $\bar{u}$ (as well as in $a$ and~$y$).
Setting $\bar{g}(\bar{\lambda})=f(\y,\mu) \in \R^n$
the system 
\begin{equation}\label{e:barprelim}
  \left\{ \begin{array}{l} \dot{\bar{u}} = \bar{f}(\bar{u},\bar{\lambda}), \\
  \bar{g}(\bar{\lambda})=0, \\
  \bar{u}(t) \text{ is $2\pi$-periodic},
  \end{array}\right.
\end{equation}
is again of the form~\eqref{eq:h_vector_field}.
What remains is to introduce appropriate phase and continuation equations, as well as an ``amplitude'' equation which lifts the invariance 
under the continuous rescaling $a \to \theta^{-1} a$ and $\bar{u}(t) \to \theta \bar{u}(t)$
for $\theta \in \R$.

We collect all the variables in $x=(v,\bar{\lambda})$,
where $v$ denotes the Fourier coefficients of $\bar{u}$.
As in Section~\ref{s:setup}
we  assume we have two points $\hx_0 = (\hv_0,\hl_0) \in \XS_K$ and $\hx_1 = (\hv_1,\hl_1) \in \XS_K$ which each represent an approximate solution
of the Fourier equivalent of~\eqref{e:barprelim}. 
While it is now hidden in the notation that we are solving~\eqref{e:barprelim} rather than~\eqref{e:prelim}, 
we use the same phase condition 
%$G^\phase(\bar{u})=0$ 
as in Section~\ref{s:setup}: 
$
  G^\phase(x)=   \bigl\langle q^\phase , x \bigr\rangle  =0 ,
%  \G^{\hq_0,\hq_1}_{\langle \hx_0 , \hq_0 \rangle,\langle\hx_1 , \hq_1 \rangle }(x,s) ,
$ 
where $q^\phase = (q^\phase_v,0)  \in X_K$ 
is given by~\eqref{e:choiceofphaseq}.
For the amplitude equation $G^\ameq(x)$ we use
\begin{equation}\label{e:amplitudeequation}
G^\ameq(x) \bydef   \bigl\langle q^\ameq , x \bigr\rangle  -1  =0 ,
% =   \G^{q^\ameq,q^\ameq}_{1,1}(x,s),
\end{equation}
where 
$q^\ameq = (q^\ameq_v,0) \in \XS_K$,
with 
%\corrc Note the new $\K^2$ and the $\hv_{\frac{1}{2}}$. Elena: has this been implemented? JB Yes. E<<>>
%\blue{not yet, working on it today
% at the moment, in the amplitude equation I have }
% $$
% \sum_{k,j} (x_k)_j(\hat x_k)_j - \sum_{k,j} (\hat x_k)_j(\hat x_k)_j =0,
% $$
% \blue{where $x$ is the periodic solution and $\hat x$ is the numerical solution. Then it should be changed to }
% $$
% \sum_{k,j} k^2(x_k)_j(\hat x_k)_j - \sum_{k,j} k^2(\hat x_k)_j(\hat x_k)_j =0,
% $$
% \blue{did I understood it right?}
\begin{equation}\label{e:K2choice}
  (q^\ameq_v)_i =  \overline{\K^2(\hv_{\frac{1}{2}})_i} \qquad\text{for }
  |k| \leq K, 1\leq i \leq n.
\end{equation}
We will slightly abuse the notation for the bilinear form to write
$\langle q^\ameq , x \rangle =  \langle q^\ameq_v , v \rangle$,
and by $[q^\ameq_v]_k \in \C^n$ we will denote the $k$-th Fourier component of $q^\ameq_v$.
%
%
%
% For the amplitude equation $G^\ameq(x)$ we use \red{TO BE ADAPTED}
% \[
%   G^\ameq_s(x)=\G^{\hp_0,\hp_1}_{\langle \hx_0 , \hp_0 \rangle,\langle\hx_1 , \hp_1 \rangle }(x,s) ,
% \]
%  \blue{or perhaps easier:
%  \[
%    G^\ameq_s(x)=\G^{\hp_0,\hp_1}_{1,1}(x,s) .
%  \]}
% {\color{green} I like the second option better, it's also the one implemented}
% where
% $\hp_\s = (0, p_\s) \in \XS_K$ for $\s=0,1$,
% with $p_\s$ given by
% \[
%   ((p_\s)_i)_k = ((\hv_\s)_i)_{-k} \qquad\text{for }
%   |k| \leq K, 1\leq i \leq N.
% \]
%
%
With the choice~\eqref{e:K2choice} the corresponding Equation~\eqref{e:amplitudeequation} represents a linear approximation of
%\corrc Notice the (new) time derivative. I think this is what we should want to use. JB <<>>
\[
  \frac{1}{2\pi} \int_0^{2\pi} \sum_{i=1}^n  \bigl|\dot{\bar{u}}_i(t) \bigr|^2 dt =1  .
\]
% assuming the numerically determined end points $\hv_{\s}$, $s=0,1$, satisfy the amplitude constraint
% $
%   \sum_{k=-K}^K \sum_{i=1}^N k^2 |((\hv_s)_i)_k|^2  \approx 1.
% $
 
Finally, the continuation equation $G^\cont_s(x)=0$ is chosen as in Section~\ref{s:setup}, see~\eqref{e:continuationequation},
where  the 
``predictors'' $\dot{\hx}_\s$ are now of the tangent direction of the solution curve for the problem including both the phase and the amplitude condition.

Denoting by $\bar{F}$ the Fourier transform of the renormalized vector field~\eqref{e:deftildef},
the full set of equations becomes
\[
\bar{H}_s(x) = 
\begin{bmatrix} \bar{F}(x)\\\bar{G}_s(x)\end{bmatrix}
\qquad
\text{with}
\qquad
\bar{G}_s(x) =
\begin{bmatrix}  G^\phase(x) \\ G^\ameq(x) 
	\\ \bar{g}(\bar{\lambda}) \\ G^\cont_s(x) \end{bmatrix} .
\]
For $\bar{H}_s(x)=0$ we then carry out the extension construction of Section~\ref{s:saddlenode}, leading to the problem
$\bar{\bs{H}}_s(x)=0$. To the latter we apply the technique from Theorem~\ref{thm:radii_bound}
to obtain a parametrized solution curve 
$\sol{\bs{x}}(s)=(\sol{\bs{v}},\sol{\bar{\bs{\lambda}}})(s)$.
The first ``block'' $\sol{\bs{x}}^{[0]}=(\sol{\bs{v}}^{[0]},\sol{\bar{\bs{\lambda}}}^{[0]})$  
corresponds to  
$\sol{\bar{\bs{\lambda}}}^{[0]}(s)=\sol{\bar{\lambda}}(s)=(\sol{\tau}(s),\sol{a}(s),\sol{y}(s),\sol{\mu}(s))$ 
and the Fourier coefficients of
$\sol{\bar{u}}(t;s)$.
The second block $\sol{\bs{x}}^{[1]}$ contains the first derivatives
$(\sol{\tau}'(s),\sol{a}'(s),\sol{y}'(s),\sol{\mu}'(s))$ and the Fourier coefficients of $\partial_s \sol{\bar{u}}(t;s)$,
while $\sol{\bs{x}}^{[2]}$ contains their second derivatives.
% {\color{green} the check for the saddle node takes place just for the segment with $a=0$ and if Remark 3.5 is not satisfied.}
% \corrc Green is unreadable! Perhaps I am colorblind. JB <<>>

Before we can properly formulate the result, we need to analyze the problem at $a=0$. In particular, we aim to establish that $\sol{a}(s_0)=0$ implies $\sol{\mu}'(s_0)=0$.
Hence suppose $\sol{a}(s_0)=0$ for some $s_0 \in (0,1)$. 
We denote the Jacobian of the equilibrium problem by
\[
  \A_0 \bydef D_u f(\sol{\y}(s_0),\sol{\mu}(s_0)).
\]
Then, since $\bar{f}$ is linear in $\bar{u}$ at $a=0$,
\begin{equation}\label{e:Fislinear}
  \bigl(\bar{F} (\sol{x}(s_0))\bigr) _{k} =
  \bigl(\imag k I_n  - \sol{\tau}(s_0) \A_0 \bigr) \sol{v}_k(s_0) = 0,
  \qquad\text{for all } k\in \Z,
\end{equation}
with $\sol{v}_k(s_0) \in \C^n$ and $I_n$ the identity matrix on $\C^n$,
i.e., the ODE is diagonalized (in $k$) in Fourier space.
Since $\langle q^\ameq_v , \sol{v}(s_0) \rangle =1$, 
and $[q^\ameq_v]_0 =0$, there must be at least one $k_0 \in \Z\setminus\{0\}$
such that 
\[
  M_{k_0} \bydef \imag k_0 I_n  - \sol{\tau}(s_0) \A_0 
\]
is a non-invertible matrix. We note that this implies that $\sol{\tau}(s_0) \neq 0$. By conjugation symmetry $M_{-k_0}$ is then non-invertible as well ($\A_0=\A_0^*$ and $ \sol{\tau}(s_0) \in \R$).

%   \Bigl( -\imag k I_n
%   + \sol{\tau}(\sstar) D_u f(\sol{v}_0(\sstar),\sol{\mu}(\sstar)) \Bigr)
%    \tilde{v}_k ,
% \]
% i.e., diagonal in $k$, 
Furthermore, collecting the variables $z=(v,\tau,a,\y)$
and the equations
\[
  \hH (z,\mu) \bydef \begin{bmatrix} \bar{F}(z,\mu) \\ G^\phase(z) \\ G^\ameq(z) \\  f(\y,\mu)\end{bmatrix} = 0 ,
\]
its Jacobian can be decomposed as
\begin{equation}\label{e:defP}
  \PP \bydef D_z \hH (\sol{x}(s_0)) =
   \begin{bmatrix}
    D_v \bar{F}(\sol{x}(s_0)) & D_\tau \bar{F}(\sol{x}(s_0)) & D_a \bar{F}(\sol{x}(s_0))  & D_y \bar{F}(\sol{x}(s_0)) \\
    q^\phase & 0 & 0& 0 \\
	q^\ameq & 0 & 0& 0\\
	0&0&0& A_0 
   \end{bmatrix} .
\end{equation}
Here, again by linearity of $\bar{f}$ at $a=a(s_0)=0$,
\[
  \bigl(D_v \bar{F}(\sol{x}(s_0)) v\bigr)_k =
  \bigl(\bar{F} (\sol{x}(s_0))\bigr) _{k} =
  \bigl(\imag k I_n  - \sol{\tau}(s_0) \A_0 \bigr) v_k,
\]
which is similar to~\eqref{e:Fislinear} because of linearity,
and
\[
 (D_\tau \bar{F}(\sol{x}(s_0)))_k =  \A_0 \sol{v}_k(s_0),
\]
and 
\begin{equation}\label{e:DabarF}
 D_a \bar{F}(\sol{x}(s_0)) =  \tfrac{1}{2} D^2_{u,u}
  \tilde{f}(\sol{\y}(s_0),\sol{\mu}(s_0)) [\sol{v}(s_0),\sol{v}(s_0)]  ,
\end{equation}
which is to be interpreted in terms of convolution products,
and
\[
 (D_y \bar{F}(\sol{x}(s_0)))_k =  D^2_{u,u} \tilde{f}(\sol{\y}(s_0),\sol{\mu}(s_0)) \sol{v}_k(s_0),
\]
which is to be interpreted as an $n\times n$ matrix for each $k\in \Z$.

Since
\begin{equation}\label{e:blockDxHs0}
  D_x H_{s_0}(\sol{x}(s_0)) = \begin{bmatrix}
   \PP & D_\mu \hH(\sol{x}(s_0)) \\
   q^\cont_z(s_0) & q^\cont_\mu(s_0)
   \end{bmatrix}
\end{equation}
is invertible, as the solution curve was obtained through Theorem~\ref{thm:radii_bound}, the kernel of the operator $\PP$ can be at most $1$ dimensional. In turn this implies that the kernel of $D_v \bar{F}(\sol{x}(s_0))$
is at most $3$ dimensional. Indeed this follows from the expression~\eqref{e:defP} for $\PP$, and in particular the $0$ in the lower left corner, which implies that, apart from $D_v \bar{F}(\sol{x}(s_0))$, there are only two nonzero rows in the left (block) column. 

In view of $M_{\pm k_0}$ being non-invertible, the dimension of the kernel of $D_v \bar{F}(\sol{x}(s_0))$ is at least $2$-dimensional. If there would be  a $0 \neq k_1 \neq \pm k_0$ such that $M_{k_1}$ is non-invertible, then $M_{-k_1}$ would be non-invertible as well, implying that the kernel of $D_v \bar{F}(\sol{x}(s_0))$ is at least $4$ dimensional, a contradiction. Hence $M_{k}$ is invertible for all $k \notin \{0,\pm k_0\}$. 
The next remark explains why we may assume $M_0 =- \sol{\tau}(s_0) D_u f(\sol{\y}(s_0),\sol{\mu}(s_0))$ to be invertible.
\begin{remark}\label{r:A0invertible}
Suppose $M_0$ is not invertible and that $\tilde{v}_0 \neq 0$ is in the kernel of $M_0$. Since $\sol{\tau}(s_0) \neq 0$, the vector $\tilde{v}_0$ is in the kernel of $A_0$, hence $(0,0,0,v_0)$ is in the kernel of $\PP$.
Furthermore, since $[q^\ameq_v]_0 =0$ and $[q^\phase_v]_0=0$, it follows that $(\tilde{V}_0,0,0,0)$ is in the kernel of $\PP$,
where $(\tilde{V}_0)_k =0$ for $k \neq 0$ and $(\tilde{V}_0)_0=\tilde{v}_0$, so that $\tilde{V}_0$ is essentially just $\tilde{v}_0$ interpreted as an element of $(\ell^1_\nu)^n$. 
We infer that the kernel of $\PP$ is at least two dimensional, a contradiction. We conclude that $M_0$ and $A_0=D_u f(\sol{\y}(s_0),\sol{\mu}(s_0))$ are invertible matrices.
It then follows from  the implicit function
theorem that $\sol{\y}(s_0)$ is part of a smooth one parameter family of equilibria
$\sy(\mu)$ of~\eqref{e:initial} with $\sy(\sol{\mu}(s_0))=\sol{\y}(s_0)$.
\end{remark}
% \begin{remark}\label{r:A0invertible}
% Suppose $M_0$ is not invertible and that $\tilde{v}_0 \neq 0$ is in the kernel of $M_0$.
% Since $[q^\ameq_v]_0 =0$ and $[q^\phase_v]_0=0$, it follows that $\tilde{V}_0=(\tilde{v}_0,0,0,0)$ is in the kernel of $\PP$,
% where $\tilde{v}_0$ is now interpreted as an element of $(\ell^1_\nu)^n$.
% By the arguments in Remark~\ref{r:kernelimpliesvertical} we find that
% $\sol{x}'(s_0)$ is a multiple of $(\tilde{v}_0,0,0,0)$.
% Such a scenario can be excluded in several ways, for example by  checking that one of the other components of $\sol{x}'(s_0)$ does not vanish, or by choosing a continuation equation such that $[(q^\cont_s)_v]_0=0$. In particular, one could check that $\sol{a}'(s_0) \neq 0$ or choose to essentially perform parameter continuation \red{in $a$ by - notation unclear} setting $q^\cont_s=(0,0, (q^\cont_s)_a,0,0)$, so that $\langle q^\cont, x \rangle = (q^\cont_s)_a a$.
% In any case, from now on we will assume that such a check or choice has been made, so that we can conclude that $M_0$ and hence
% $\A_0=D_u f(\sol{\y}(s_0),\sol{\mu}(s_0))$ is invertible.
% It thus follows from  the implicit function
% theorem that $\sol{\y}(s_0)$ is part of a smooth one parameter family of equilibria
% $\sy(\mu)$ of~\eqref{e:initial} with $\sy(\sol{\mu}(s_0))=\sol{\y}(s_0)$.
% \end{remark}

Since $M_k$ is invertible for $k \neq \pm k_0$, while $M_{k_0}$ and $M_{-k_0}$
have one-dimensional kernels (related by conjugation), it follows
from~\eqref{e:Fislinear} that $\sol{v}_k(s_0)=0$ for $k \neq \pm k_0$, whereas 
$\sol{v}_{\pm k_0}(s_0) \neq 0$ in view of $\langle q^\ameq_v , \sol{v}(s_0) \rangle =1$. 
For definiteness we will from now on, without loss of generality, assume that $k_0=1$.  
This corresponds to the linearized problem at the equilibrium $\sol{\y}(s_0)$ having purely imaginary eigenvalues $\pm \imag \sol{\tau}(s_0)^{-1}$.
Note that if $k_0>1$ then we may simply replace $\tau$ by $\tau/k_0$ (and thus reduce to the minimal period).
By the arguments above, the kernel of $M_1$ and $M_{-1}$ is  one-dimensional.
% Furthermore, it follows from Remark~\ref{r:A0invertible} that $D_u
% f(\sol{\y}(s_0),\sol{\mu}(s_0))$ is invertible, hence by the implicit function
% theorem $\sol{\y}(s_0)$ is part of a smooth one parameter family of equilibria
% $\sy(\mu)$ of~\eqref{e:initial} with $\sy(\sol{\mu}(s_0))=\sol{\y}(s_0)$.

In the following we will construct an element in the kernel of $\PP$. This implies, using again the arguments in  Remark~\ref{r:kernelimpliesvertical}, that $\sol{\mu}'(s_0)=0$. 
%
% The next observation is somewhat similar to the previous remark, and is analogous to Remark~\ref{r:notphaseshift}.
% \begin{remark}\label{r:notshiftamplitude}
% We will assume that we have properly set up the problem in such a way that the time-shift as well as the invariance under rescaling $\sol{v}  \to \theta \sol{v}$ and $\sol{a} \to \theta^{-1} \sol{a}$ do not cause degeneracy. To be precise, consider the $2\times 2$ matrix
We define the $2\times 2$ matrix
\begin{equation}\label{e:matrixC}
   C \bydef \begin{bmatrix} 
	\langle [q^\phase_v]_{-1} , \sol{v}_{-1}(s_0) \rangle &
	\langle [q^\phase_v]_{1} , \sol{v}_{1}(s_0) \rangle \\
	\langle [q^\ameq_v]_{-1} , \sol{v}_{-1}(s_0) \rangle &
	\langle [q^\ameq_v]_{1} , \sol{v}_{1}(s_0) \rangle 
	 \end{bmatrix} .
\end{equation}
There are two cases to consider: $C$ is invertible or not. We start with the latter case.

If $C$ is non-invertible, say $0 \neq [c_-,c_+]^T$ is in its kernel,
then it is easily seen that $(\tilde{v},0,0,0)$  is in the kernel of $\PP$,
where $\tilde{v}_{k}=0$ for $k\neq \pm 1$, and
$\tilde{v}_{\pm 1} = c_{\pm} \sol{v}_{\pm 1}(s_0)$.
Hence by the arguments in Remark~\ref{r:kernelimpliesvertical}
$\sol{x}'(s_0)$ is a multiple of $(\tilde{v},0,0,0,0)$, and $\sol{\mu}'(s_0)=0$.
As a side remark, this direction can be interpreted as a combination of time-shift and rescaling $a$. 
If desired, this scenario can easily be avoided by choosing 
 $q^\phase$ and $q^\ameq$ appropriately or, alternatively,
 % is such that
% (q^\phase_v)_{\pm k_0}  \sol{v}_{\pm k_0}(s_0) \approx \red{???}$ and
% (q^\ameq_v)_{\pm k_0}  \sol{v}_{\pm k_0}(s_0) \approx \red{???}$,
% o that \red{what????}. similarly to \red{what???}, one may implement a check \red{what???}  to prove that $C$ is invertible. Alternatively,
by checking that $\sol{a}'(s_0) \neq 0$, cf.~Remark~\ref{r:A0invertible}.

% \red{JB: While it does provide insight, it is not entirely clear to me why this would be important; if $C$ is non-invertible we nevertheless find $\sol{\mu'(s_0)}=0$, which is really the only thing we need, since it already implies $\sol{a}(s)\neq 0$ for other $s$ by the the parabolicity of $\sol{\mu}(s)$. Indeed in all three scenarios (including Remark~\ref{r:A0invertible}) we have $\sol{\mu}'(s_0)=0$ and parabolicity does the rest. We do need $D_u f(\sol{\y}(s_0),\sol{\mu}(s_0))$ invertible. But $\sol{a}'(s_0) \neq 0$ is somehow just for our mental picture.}

If $C$ is invertible, then we construct an element of the form $(\tilde{v},0,1,0)$ in the kernel of $\PP$. This implies, using again the arguments in  Remark~\ref{r:kernelimpliesvertical}, that $\sol{x}'(s_0)$ is a multiple of $(\tilde{v}_0,1,0,0)$ and $\sol{\mu}'(s_0)=0$, in particular. 
%\begin{remark}\label{r:constructtangent}
Since $M_k$ is invertible for $k \neq \pm 1$, and $M_1$ and $M_{-1}$ have one-dimensional kernels (related by conjugation), it follows from~\eqref{e:Fislinear} that $\sol{v}_k(s_0)=0$ for $k \neq \pm 1$ and
$\sol{v}_{\pm 1}(s_0) \neq 0$. We infer from~\eqref{e:DabarF} and the properties of the convolution product that
$D_a \bar{F}_k(\sol{x}(s_0))$ vanishes for all $k \notin \{-2,0,2\}$.
We set 
\[
  \tilde{v}_k  =  - M_k^{-1} D_a \bar{F}_k(\sol{x}(s_0))
  \qquad\text{for } k \in \{-2,0,2\},
\]
and $\tilde{v}_k =0$ for $k \notin \{-2,-1,0,1,2\}$.
Next we set $\tilde{v}_{\pm 1}(s_0) = c_{\pm}\sol{v}_{\pm 1}(s_0) $,
where 
%$c_\pm$ are chosen such that
\begin{equation*}
 \begin{bmatrix} c_- \\ c_+ \end{bmatrix} =
	- C^{-1}
  \begin{bmatrix}  
	\langle [q^\phase_v]_{-2} , \tilde{v}_{-2}\rangle 
  + \langle [q^\phase_v]_{0} , \tilde{v}_{0}\rangle
  + \langle [q^\phase_v]_{2} , \tilde{v}_{2}\rangle  \\
    \langle [q^\ameq_v]_{-2} , \tilde{v}_{-2}\rangle 
  + \langle [q^\ameq_v]_{0} , \tilde{v}_{0}\rangle
  + \langle [q^\ameq_v]_{2} , \tilde{v}_{2}\rangle
  \end{bmatrix} ,
\end{equation*}
with $C$ defined in~\eqref{e:matrixC}.
It is not difficult to check that, by construction, $(\tilde{v},0,1,0)$
is in the kernel of $\PP$.
%\end{remark}

We are now ready to state a result for the rigorous verification of Hopf bifurcations.
% \corrc Shall we call this a proposition? JB <<>>
\begin{proposition}\label{prop:hopf}
Assume
\begin{subequations}
\label{e:signcondhopf}
\begin{gather}
  \big(\hat{\bs{\mu}}^{[1]}_0\bigr)_j + \hr < 0 , \quad\text{and }\quad
  \bigl(\hat{\bs{\mu}}^{[1]}_1\bigr)_j - \hr > 0 , \label{e:muendpoints}\\
  \min\Bigl\{\bigl(\hat{\bs{\mu}}^{[2]}_0\bigr)_j,
	  		 \bigl(\hat{\bs{\mu}}^{[2]}_1\bigr)_j\Bigr\} -  \hr > 0 ,
			 \label{e:musecondpositive} \\
  \big(\hat{\bs{a}}^{[0]}_0\bigr)_j + \hr < 0 ,   \quad\text{and }\quad
  \bigl(\hat{\bs{a}}^{[0]}_1\bigr)_j - \hr > 0 . \label{e:aendpoints}
%   \min\Bigl\{\bigl(\hat{\bs{a}}^{[1]}_0\bigr)_j,
% 	  		 \bigl(\hat{\bs{a}}^{[1]}_1\bigr)_j\Bigr\} -  \hr > 0 .
% 			 \label{e:afirstpositive}
\end{gather}
\end{subequations}
Then the solutions curve $\sol{\bs{x}}(s)$ goes through a unique nondegenerate Hopf bifurcation (folding to the right) with respect to $\mu$ in the interval $s\in [0,1]$.
\end{proposition}
\begin{proof}
The proof follows the same lines as the one of Proposition~\ref{prop:fold},
and we comment only on the additional steps.
Assumptions~\eqref{e:muendpoints} and~\eqref{e:musecondpositive}
imply that there is a unique $\sstar \in (0,1)$ such that $\sol{\mu}'(\sstar)=0$. 
Assumption~\eqref{e:aendpoints}
implies that there is a $s_0 \in (0,1)$ such that $\sol{a}(s_0)=0$. 
The analysis of the kernel of the operator $\PP$ above shows that 
$\sol{\mu}'(s_0)=0$.
Hence the unique nondegenerate fold with respect $\mu$ occurs at $\sstar=s_0$.

For $s \neq \sstar$ we have $\sol{a}(s) \neq 0$,
since the analysis above shows that $\sol{a}(s)=0$ 
 implies $\sol{\mu}'(s)=0$, which occurs at $s=\sstar$ only.
%(it also this follows from~\eqref{e:afirstpositive} of course).
The amplitude condition~\eqref{e:amplitudeequation} then guarantees that for $s\neq\sstar$ the solution $\sol{u}(t;s)$, represented in Fourier space by $\sol{\y}(s)+\sol{a}(s) \sol{v}(s)$, is not time-independent.
% \corrc but $\sol{a}'(s)\neq 0$ at all times? Elena. I made a typo; anyway, the proof has been simplified, due to having an improved Remark~\ref{r:A0invertible}. JB <<>>

Finally, it follows from Remark~\ref{r:A0invertible}
that $D_u f(\sol{\y}(\sstar),\sol{\mu}(\sstar))$
is invertible, hence by the implicit function theorem $\sol{\y}(\sstar)$ is
part of a smooth one parameter family of equilibria $\sy(\mu)$ with
$\sy(\sol{\mu}(\sstar))=\sol{\y}(\sstar)$.
% Finally, it follows from Remark~\ref{r:A0invertible}
% and~\eqref{e:afirstpositive} that $D_u f(\sol{\y}(\sstar),\sol{\mu}(\sstar))$
% is invertible, hence by the implicit function theorem $\sol{\y}(\sstar)$ is
% part of a smooth one parameter family of equilibria $\sy(\mu)$ with
% $\sy(\sol{\mu}(\sstar))=\sol{\y}(\sstar)$.
\end{proof}

\begin{remark}
The conclusion of Proposition~\ref{prop:hopf} also holds when
the inequalities in the assumptions~\eqref{e:aendpoints}
%--\eqref{e:afirstpositive} 
are reversed. Similarly, when the three inequalities in~\eqref{e:muendpoints}--\eqref{e:musecondpositive} are reversed
then the result holds with the curve folding to the left.  
\end{remark}

% \begin{remark}
% Assumption~\eqref{e:afirstpositive} can be weakened, since it is only used in the proof to exclude the degenerate scenario where $D_u f(\sol{\y}(\sstar),\sol{\mu}(\sstar))$ would be non-invertible,
% which can also be excluded via alternative checks, see
%  Remark~\ref{r:A0invertible},
% or through a direct interval arithmetic computation.
% \end{remark}

Traditionally, a Hopf bifurcation point (rather than the branch of periodic solutions emanating from it) is described in terms conditions on eigenvalues of the Jacobian at the critical point as well as other normal form parameters.
Clearly, finding just the Hopf bifurcation point only requires solving the algebraic system (where we have split complex eigenvectors in real and imaginary parts, and $\varphi_1 , \varphi_2 \in \R^n$)
\[
  \begin{bmatrix}  
	 f(\y,\mu) \\
	 D_u f(\y,\mu) \y_1 + \beta \y_2 \\
	 D_u f(\y,\mu) \y_2 - \beta \y_2 \\
	 \varphi_1^T \y_1 - \varphi_2^T \y_2 - 1\\
	 \varphi_2^T \y_1 + \varphi_1^T \y_2 
  \end{bmatrix}
  =0,  \qquad \text{with } (\y,\y_1,\y_2,\mu,\beta) \in \R^{3n+2}, 
\]
% \corrc Is this a reasonable choice for the ``phase'' conditions for the eigenvector? JB <<>>
% \blue{it is the one used on the algebraic code and in practise it works most of the times. In the code $\phi_i$ are set random until a pair is found to work, then that pair is (manually) stored and used the following times. Some systems take some tries before working but up to now it always worked out.}
which can also be attacked using the radii polynomial approach, albeit in the much simpler finite dimensional setting.
Obviously, this is how one may locate computationally a Hopf bifurcation point
and use it as a numerical starting point for a rigorous continuation of the desingularized problem for the periodic orbits.
Additionally, if desired, one may study the eigenvalue problem of the equilibrium ($\varphi_0 \in \C^n$)
\[
   \LL_\mu (\y,\y_0,\gamma)
   \bydef
  \begin{bmatrix}  
	 f(\y,\mu) \\
	 D_u f(\y,\mu) \y_0 -\gamma \y_0 \\
	  \varphi_0^T  \y_0  - 1
  \end{bmatrix}
  =0,  \qquad \text{with } (\y,\y_0,\gamma) \in \C^{2n+1}, \mu \in \R,
\] 
for example to determine the algebraic multiplicity of the purely imaginary eigenvalues 
$\imag\sol{\tau}(\sstar)^{-1}$ and to confirm that 
$\A_0 = D_u f(\sy(\mu_{\sstar}),\mu_{\sstar})$ 
has no other purely imaginary eigenvalues apart from its complex conjugate (the analysis in Fourier space excludes only integer multiples of 
$\pm \imag\;\!\sol{\tau}(\sstar)^{-1}$). Furthermore, when the algebraic multiplicity of $ \imag\;\!\sol{\tau}(\sstar)^{-1}$ is one (as it generically will be) then one may establish the direction in which this eigenvalue moves when $\mu$  is varied through $\mustar=\sol{\mu}(\sstar)$, by solving the linear system
\[
  \begin{bmatrix}  
  \A_0 & 0 & 0 \\
  D^2_{u,u} f(\sy(\mustar),\mustar) \sy_0(\mustar) &
  \A_0 - \stat{\gamma}(\mustar) 
  & -\sy_0(\mustar) \\
  0 & \varphi_0 & 0 \\
  \end{bmatrix} 
  \begin{bmatrix}  
  	\sy'(\mustar) \\ \sy_0'(\mustar) \\ \stat{\gamma}'(\mustar) 
  \end{bmatrix} 
	= 
  \begin{bmatrix}  
  	D_\mu f(\sy(\mustar),\mustar) \\
    D^2_{u,\mu} f(\sy(\mustar),\mustar)      
	                              \sy_0(\mustar) \\
	0
  \end{bmatrix} 
  .
\] 
Here the matrix in the lefthand side is invertible when $\A_0$ is invertible and the algebraic multiplicity of 
$\stat{\gamma}(\mustar) = \imag\sol{\tau}(\sstar)^{-1}$ is $1$.
All these computations are on finite dimensional algebraic systems, and can relatively easily be done in interval arithmetic to ensure that the results are mathematically rigorous.

%\section{Refinement in saddle segment}
%\label{s:refinement}
%\input{refinement}

\section{From Hopf bifurcation to global continuation}
\label{s:gluing}
%!TEX root = bifcont.tex

% gluing of the Hopf system to a more general continuation system

In the previous sections, we presented an approach to validate Hopf bifurcations and a local family of periodic orbits.
In the desingularized Hopf problem \eqref{e:mainhopf} we solved for $\bar{u}$ and $\bar{\lambda} = (\tau, a, y, \mu)$ instead of  $u$ and $\lambda=(\tau,\mu)$.
At some distance from the Hopf bifurcation point one would prefer to start working directly with the simpler systems \eqref{e:main}. 
%
%  that originate from them. In this section, we discuss how to step away from the desingularized Hopf system and get back to the general system
%
% In particular, in order to validate the Hopf bifurcation, instead of solving for $u$ and $\lambda$, we had to solve for .
Here we discuss how to switch from a solution branch for $(\bar u, \bar \lambda)=(\bar{u},(\tau, a, y, \mu))$ to a solution branch for $(u,\lambda)=(u,(\tau,\mu))$. In particular, when we ``glue'' the end (periodic orbit) point of a continuation step for the desingularized system to the starting (periodic orbit) point of a continuation step for the original system, we want to be sure that the solutions branches connect.
We denote the numerical approximations at the boundary points by
$(\hat{\bar{u}}_1,(\hat{\tau}_1, \hat{a}_1, \hat{y}_1, \hat{\mu}_1))$
and $(\hat{u}_0,(\hat{\tau}_0, \hat{\mu}_0))$, respectively,
and for natural reasons we choose to set 
\begin{equation}\label{e:gluecenters}
	\hat{\tau}_0=\hat{\tau}_1  \quad\text{and}\quad \hat{\mu}_0=\hat{\mu}_1
\quad\text{and}\quad  \hat{u}_0 = \hat{y}_1 + \hat{a}_1 \hat{\bar{u}}_1 .
\end{equation}
The solutions found at the boundary points are denoted by
$(\sol{\bar{u}}_1,(\sol{\tau}_1, \sol{a}_1, \sol{y}_1, \sol{\mu}_1))$
and
$(\sol{u}_0,(\sol{\tau}_0, \sol{\mu}_0))$. For the solutions curves to glue nicely we need to check that 
\begin{equation}\label{e:gluesuccess}
\sol{\tau}_0=\sol{\tau}_1 \quad\text{and}\quad \sol{\mu}_0=\sol{\mu}_1
\quad\text{and}\quad \sol{u}_0 = \sol{y}_1 + \sol{a}_1 \sol{\bar{u}}_1 .
\end{equation}
The main technical issue lies in lining up the phase condition and continuation equation at the boundary points. There is considerable freedom in choosing these equations, and we will make use of that.
In what follows we will switch from a function $u$ to its Fourier components $v$ without further ado.

We note that the coordinate transformation
\begin{equation}\label{e:trans}
u = y + a \bar{u}
\end{equation}
is essentially a nonlinear change of variables, since $a$ is part of the set of unknowns. On the other hand, in terms of Fourier coefficients the transformation is relatively simple: all modes get rescaled by the same scalar $a$ and only in the $0$-th mode the vector $y$ is added.
Let us denote the phase condition at the starting point of the continuation step in the original problem by
\[
\langle q^\phase_{0}, x \rangle = 0, \qquad\text{where } q^\phase_{0}=(q^\phase_{0v},0). 
\]
Here we assume the $0$-th Fourier component of $q^\phase_{0v}$ to vanish, see also~\eqref{e:choiceofphaseq}.  
In view of the action of the transformation~\eqref{e:trans} on the Fourier coefficients, as discussed above, it transforms~\eqref{e:gluecont0} into an \emph{equivalent} condition in desingularized coordinates of the form
\begin{equation}\label{e:gluephase}
  \langle \bar{q}^\phase_{1} , \bar{x} \rangle =0,
\qquad\text{where } \bar{q}^\phase_{1}=(\bar{q}^\phase_{1v},0)
\quad\text{with } \bar{q}^\phase_{1v}=\hat{a}_1 q^\phase_{0v},
\end{equation}
provided $\hat{a}_1 \neq 0$. Other rescalings work as well; this particular one is inspired by~\eqref{e:choiceofphaseq}. 
Hence we choose~\eqref{e:gluephase} as the phase condition at the end of the continuation step in the desingularized problem, i.e.~$\bar{q}^\phase_{1v}=\hat{a}_1 q^\phase_{0v}$.
In essence, this guarantees that the phase of the solution does not shift at the transition. 
% We remark that, equivalently, the argument can also be made in the opposite
% direction, i.e., going from the desingularized to the original problem.

Next, we follow a similar reasoning for the continuation equation.
Let us denote the continuation equation at the starting point of the continuation step in the original problem by
\begin{equation}\label{e:gluecont0}
\langle q^\cont_0, x  \rangle = \langle q^\cont_0, \hat{x}_0  \rangle .
\end{equation}
In general the transformed condition in desingularized coordinates is
\emph{not} affine linear. In principle, this is not a problem for the
continuation method. The restriction to affine linear conditions was only made
for simplicity of presentation in the current paper. Preferring to stay in this
affine linear context for consistency, we simply require
$q^\cont_0=(q^\cont_{0v},q^\cont_{0\lambda})$ to have nonvanishing
$\lambda$-components only, i.e.\
$q^\cont_0=(0,q^\cont_{0\lambda})=(0,(q^\cont_{0\tau},q^\cont_{0\mu}))$. This
essentially corresponds to parameter continuation rather than 
pseudo-arclength continuation at this gluing step.
For this choice, the corresponding equivalent condition for the desingularized problem is
\begin{equation}\label{e:gluecont}
\langle \bar{q}^\cont_1, \bar{x}
  \rangle = \langle \bar{q}^\cont_1, \hat{\bar{x}}_0  \rangle  ,
  \qquad\text{where }\bar{q}^\cont_1=(0,(q^\cont_{0\tau},0,0,q^\cont_{0\mu})).
\end{equation}
Assuming parameter continuation at the transition point requires us to ``force'' the continuation code to switch away from the preferable pseudo-arclength continuation in the neighborhood of the gluing point. Nevertheless, this can be implemented in a relatively straightforward manner. 

With the choices~\eqref{e:gluephase} and~\eqref{e:gluecont} for the phase and continuation equations, it is not difficult to establish that
$
  ( \sol{y}_1 + \sol{a}_1 \sol{\bar{u}}_1 , (\sol{\tau}_1,\sol{\mu}_1))
$
is a solution to the problem at the starting point of the continuation step for the original system. It remains to establish that it is the same solution as
$(\sol{u}_0,(\sol{\tau}_0,\sol{\mu}_0))$.
For this final step we use the uniqueness result in Theorem~\ref{thm:radii_bound}.
Let the balls used to prove the end and starting points be denoted by $B_1=B_{\hr_1}(\hat{\bar{v}}_1,\hat{\bar{\lambda}}_1)$
and $B_0=B_{\hr_0}(\hat{v}_0,\hat{\lambda}_0)$, respectively.
If the transformed ball
\[
  B_{1 \to 0} \bydef \{ (y+av ,(\tau,\mu)) : (v,(\tau,a,y,\mu)) \in B_1 \}
\] 
and the ball $B_0$ are nested, then by uniqueness of the solutions in $B_1$ and $B_0$ we conclude that indeed \eqref{e:gluesuccess}~hold, see also~\cite{breden-vanicat,BLM} for similar arguments. Since the centers of the balls are equivalent in view of~\eqref{e:gluecenters},
to guarantee the inclusion $B_{1 \to 0} \subset B_0$
it suffices to check that 
%\red{Something like this;  didn't check really.}
\begin{equation}\label{e:r0r1}
   \Bigl(1 + |\hat{a}_1| + \max_{1\leq j \leq n} \|(\hat{\bar{v}}_1)_j\|_{\nu} +  \hr_1 \Bigr) \hr_1 <  \hr_0.
\end{equation}
Indeed, let 
\[
(v,(\tau,a,y,\mu)) =
 (\hat{\bar{v}}_1,(\hat{\tau}_1,\hat{a}_1,\hat{y}_1,\hat{\mu}_1))
 + \hr_1 (\element_v,(\element_\tau,\element_a,\element_y,\element_\mu))
\]
with
\begin{equation}\label{e:unitball}
\max_{1\leq j \leq n} \|(\element_v)_j\|_\nu \leq 1, 	
\quad |\element_\tau| \leq 1, 
\quad |\element_a| \leq 1, 
\quad \max_{1\leq j \leq n} |(\element_y)_j| \leq 1, 
\quad |\element_\mu| \leq 1, 
\end{equation}
represent any element in $B_1$.
Then in view of~\eqref{e:gluecenters} we have
\[ 
(y+av ,(\tau,\mu)) - (\hat{v}_0,(\hat{\tau}_0, \hat{\mu}_0))
= (\hr_1 \element_y+  \hr_1 \hat{a}_1 \element_v 
+  \hr_1  \element_a \hat{\bar{v}}_1 + \hr_1^2 \element_a \element_v , 
(\hr_1 \element_\tau, \hr_1 \element_\mu)).
\]
The bounds~\eqref{e:unitball} imply that to check that $B_{1 \to 0} \subset B_0$
we require $\hr_1 \leq \hr_0$, which follows from~\eqref{e:r0r1}, as well as, for any $j=1,\dots,n$,
\begin{alignat*}{1}
\bigl\| \hr_1 (\element_y)_j+  \hr_1 \hat{a}_1 (\element_v)_j 
+  \hr_1  \element_a (\hat{\bar{v}}_1)_j + \hr_1^2 \element_a (\element_v)_j 
\bigr\|_\nu  \hspace*{-3cm}  \\
&\leq
\hr_1 \| (\element_y)_j \|_\nu +  \hr_1 |\hat{a}_1| \| (\element_v)_j \|_\nu
+  \hr_1  |\element_a| \| (\hat{\bar{v}}_1)_j \|_\nu + \hr_1^2 |\element_a| \| (\element_v)_j \|_\nu 
\\ &\leq
\hr_1  +  \hr_1 |\hat{a}_1| 
+  \hr_1  \| (\hat{\bar{v}}_1)_j \|_\nu + \hr_1^2 \leq \hr_0,
\end{alignat*}
where the final inequality is guaranteed by condition~\eqref{e:r0r1}.

Finally, it is clear from~\eqref{e:r0r1} that the inclusion $B_{1 \to 0} \subset B_0$, and hence continuity of the solution branch at the switching point, is more easily established when the radius $\hr_1$ used to validate the end point of the continuation in desingularized variables is taken as small as possible, while the radius $\hr_0$ used to validate the starting point of the continuation in the original variables is taken as large as possible. 

\section{Examples}
\label{s:examples}
%!TEX root = bifcont.tex

%\corrc Elena: all figures where solutions are plotted versus time still need to be redone, since the time scale was wrong, due to the $\omega$ versus $\tau$ mixup. 
%The original time variable $t$ should run from $0$ to $2 \pi \tau$. JB <<>>

%\corrc Elena: to be consistent and scientifically sound, I propose to make sure all "numerical intervals" are of the form
%$q \in 1.234786 + [-4.5,4.5] \cdot 10^{-5}$, where the float $1.234786$ has six 
%digits after the decimal seperator, corresponding nicely to the double-digit accuracy of the interval with $-5$ as the power. Due to rounding this may (in some cases, but we could just do all cases to stay on the safe side) require adding 1 to the decimal in the interval (4.6 instead of 4.5 in the example). JB <<>>
%\corrc  checked all decimal bounds, they are coherent now. E<<>>
%\corrc Yes, it looks much better; still one decimal too many in all the mid points, I think, but just leave it for now. JB <<>>
%\corrc For readibility, for now I have removed all red associated to the two issues raised in the previous remarks. JB <<>>

We present some examples that can be analyzed with the material presented in this paper, which provides a robust and flexible method for identifying and validating fold and Hopf bifurcations in systems of ODEs.
The first example is the saddle-node validation in the so-called Rychkov system. 
This is followed by three examples of Hopf bifurcations, including a Hamiltonian  problem in Section~\ref{s:hamiltonianexample}.

%%  We are interested in showing the stability and flexibility of this method
%% in identifying and validating saddles as well as Hopf bifurcations.
%%
%% We show various examples of applications. 
%The second is the validation of the Hopf normal form. Clearly, this is a a very well known system, but it is presented here just as a test case for our search and validation of Hopf bifurcations.

% The third example is a 4 dimensional hypercahotic system derived from the Lorenz-84 model and depends on 6 paramters. In the paper \cite{KutznetsovMejier}, the existence of a Hopf bifurcation branch in a 2 dimensional parameter plane is proven.
% Here we apply our method to validate the Hopf bifurcation while fixing 5 paramenters with the same values as in \cite{KutznetsovMejier}.

\subsection{The fold in the Rychkov system}
\label{s:rychkovexample}

% \begin{figure}
% \begin{center}
% \includegraphics[scale=0.4]{rychkov_solutions-eps-converted-to.pdf}
% \caption{First validated orbit (in blue), last validated orbit (in red) and orbit corresponding to the biggest parameter value $\hat\lambda^*$ (in green) in the Rychkov system, \eqref{rychkov}.}\label{f:orbitRychkov}
% \end{center}
% \end{figure}

The Rychkov system, first presented in \cite{rychkov}, is given by
\begin{equation}\label{rychkov} 
\begin{cases}
\dot u_1 = u_2 - u_1^5 +u_1^3+\mu u_1\\
\dot u_2 = -u_1.
\end{cases}
\end{equation}
It was proven in~\cite{rychkovnew} to have no periodic solutions for $\mu>0.2249654$ and two periodic solutions for $\mu<0.224$. Additionally, it was shown in \cite{rychkovnew} that the curve of periodic orbits undergoes a saddle-node bifurcation for some $\mu$ in the interval $[0.224,0.2249654]$. 
Here we locate the bifurcation point more precisely, and we prove that a \emph{nondegenerate} fold occurs.

\begin{figure}
\begin{center}
\centering
\includegraphics[width=0.6\textwidth]{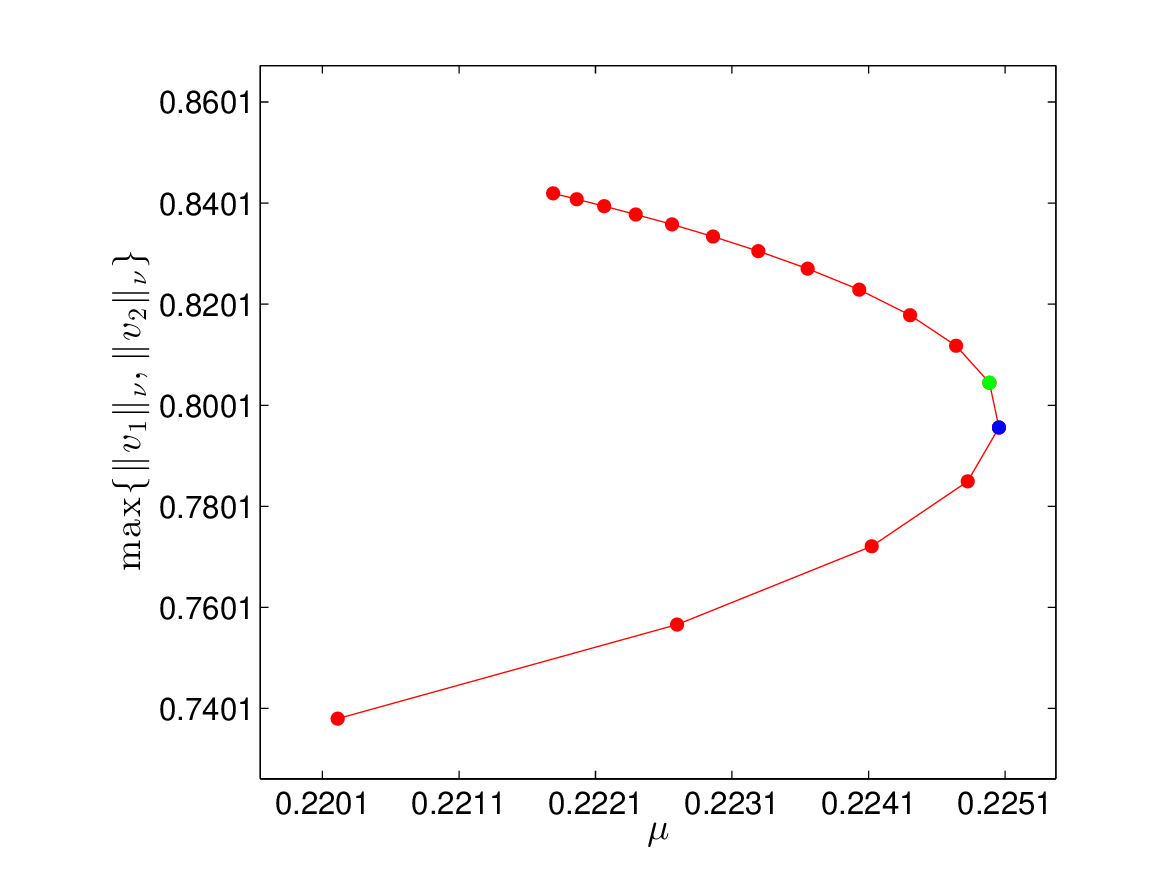}
\centering
\caption{Validated continuation of the branch of periodic solutions in the Rychkov system~\eqref{rychkov}.  The fold occurs at $%\mustar \in 0.224960422 + [-1.7,1.7] \cdot 10^{-5}$. 
\mustar \in 0.224955+ [-6.2,6.2]\cdot 10^{-5}$.
The saddle has been validated between the blue and the green dot, but this precision was possible thanks to refinement.
% The parameter $\mu$ is on the horizontal axe, the norm of the periodic solution on the vertical axe.
%\red{What is the information in the figure on the right? For me, it just shows circles. JB}
%In the second Figure we can see the two orbits associated with $\mu = 0.2245$ in blue and green and the orbit at the saddle-node bifurcation in red, for $\mu= \mu^* = 0.22496 \pm 1.7\cdot 10^{-5}$.
%\red{Elena:  What is red, what is blue? What is on the axes? JB <<>>}
}
\label{f:normRychkov}
\end{center}
\end{figure} 

%\corrc ELena: then the diagram in Figure~\ref{f:normRychkov} should also start at $\mu=0.22$ I would say. JB <<>>
%\corrc Can we also have $\max\{\|v_1\|_\nu, \|v_2\|_\nu\}$ as the $y$-label in the figure? JB <<>>
%\corrc Elena: which one? JB <<>>

We start our validated continuation at $\mu = 0.2245$
and continue the periodic solution for increasing~$\mu$. 
In Figure \ref{f:normRychkov} we plot the norm $\max\{\|v_1\|_\nu, \|v_2\|_\nu\}$
of the Fourier series of the periodic orbit versus the parameter $\mu$. 
During the continuation, we test numerically for the existence of a saddle-node, and
when a numerical indication is found for the occurence of a fold, then we validate the numerically found saddle-node using Proposition~\ref{prop:fold}.
We find that a nondegenerate fold is located at
$$
\mustar \in 0.224955+ [-6.2,6.2]\cdot 10^{-5}.
$$
For the validation of the fold, we set $\nu = 1.01$ and the code (heuristically) selects the dimension of finite dimensional projection to be $K=44$.
%This system is interesting from a computational point of view because the segment that includes the fold in Figure \ref{f:normRychkov} also includes a point where $\mu''(s)$ is zero. For this reason, a further refinement of the validation segment was necessary to verify all the conditions~\eqref{e:signcondsadnod}.

%\corrc I still don't know whether the code checks that no other saddle-nodes (w.r.t. $\mu$) occurs along the computed and validated branch. JB \blue{the code can check that, I usually don't because it's ETREMELY slow, but if I remember correctly for this validation it has been done. E} Let's discuss. JB
% \blue{update: the code does check for absence of saddle constantly, some time ago I optimised it and forgot about it. I'm rerunning the code just to be sure it returns the same figure, but it's decently fast $\approx$ some half hour with check for absence of saddle (that is now mandatory) E}
%<<>>
%
%\corrc update figure 2: more symmetric, update text with it, runs in one hour + <<>>

%This system is interesting because the second derivative of $\mu(s)$ changes sign relatively close to the saddle-node. Therefore, we had to compute and validate middle steps between to continuation solutions in order to prove that $\mu''(s)$ was not changing sign at the saddle-node.

The corresponding \textsc{matlab} code, available at~\cite{codehopf}, is provided in
\verb+figure_rychkov_saddle.m+.
%\corrc
%Elena: I think as additional information you should give the values of $K$ and $\nu$ used. Also: can you point the reader to the specific m-file that they can run to redo this proof? JB <<>>

% {\color{green}For the validation of the sadlle-node in the Rychkov system,
% validating directly the validated interval provided by the continuation
% procedure was not viable, since the validation bounds were too big to confirm
% the sadlle-node conditions as presented in Theorem \ref{validatedSaddle}. We
% therefore needed a refinement procedure that would allow us to have tighter
% bounds on subsegments of the initial branch segment. An interesting question to
% ask ourselves is how many refinements do we need of the starting segment. }

\subsection{A Hopf bifurcation in an extended Lorenz-84 model}
\label{s:lorenz84}

The extended Lorenz-84 model 
\begin{equation}\label{lorenz84}
\begin{cases}
\dot u_1 = -u_2^2-u_3^2 - a u_1 - a f - b u_4^2\\
\dot u_2 = u_1u_2 - c u_1u_3 - u_2 + d\\
\dot u_3 = c u_1u_2 + u_1u_3 - u_3\\
\dot u_4 = -e u_4 + b u_4u_1 + \mu
\end{cases}
\end{equation}
is a four dimensional system of ODEs with 7 parameters,
see \cite[Section~3.1]{KuznetsovMeijer} and \cite[Section~4.2]{KuznetsovMeijervanVeen}.
%\corrc Elena: I think there are 7 rather than 6, right?! JB <<>>
Inspired by the parameter choices in those papers, we fix 
\begin{align}\label{e:parametervalues}
  a &= 0.25,&
  b &= 0.987,&
 c &= 1, &
 d &= 0.25, &
 e &= 1.04, &
 f &= 2,
\end{align}
and consider $\mu$ as the bifurcation parameter.
The system undergoes two Hopf bifurcations at $\mu\approx 0.05$ and $\mu\approx 0.01$.

\begin{figure}
\begin{center}
\includegraphics[width = \textwidth]{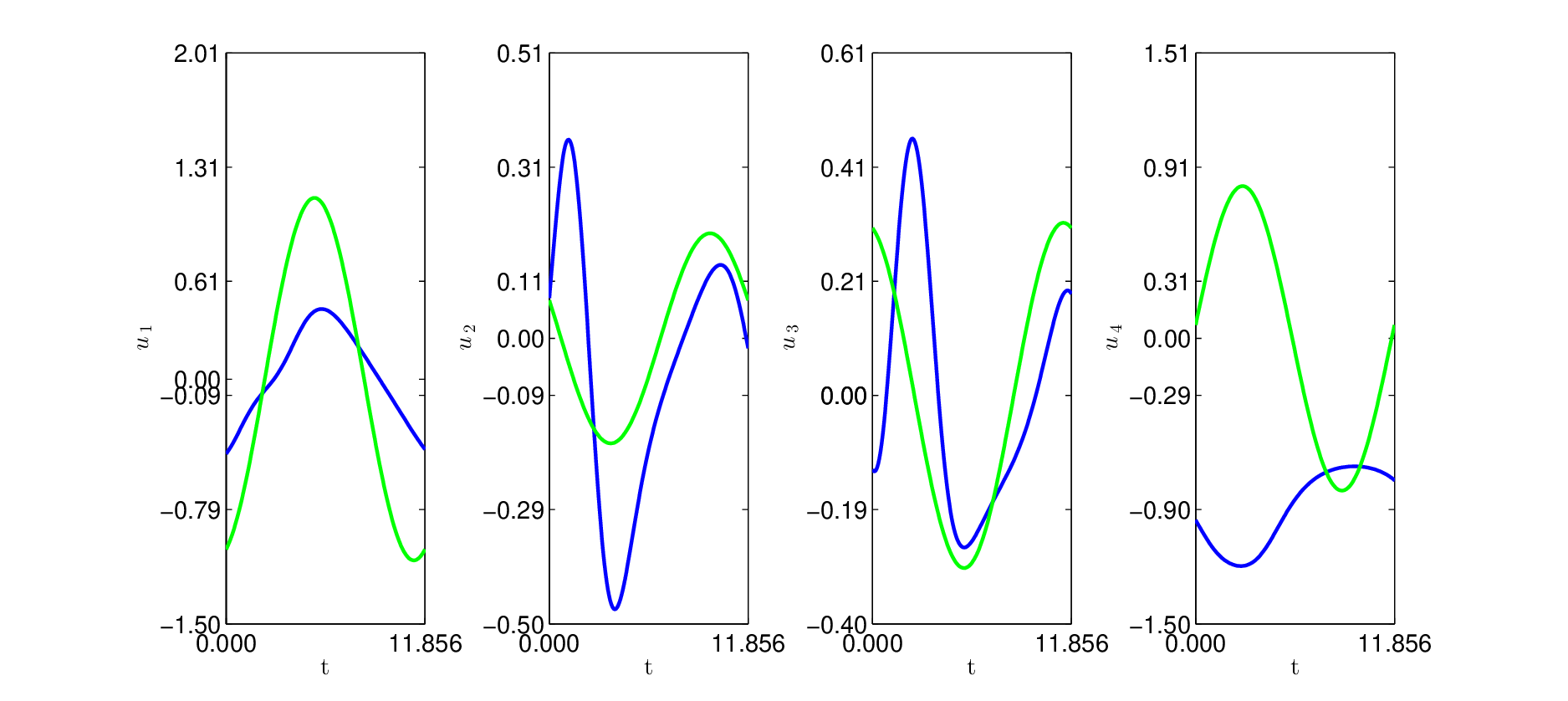}
\caption{The four components of the desingularized periodic solution $\bar u$ corresponding to the extended Lorenz-84 system~\eqref{lorenz84}, in green for $\mu\approx 0.056$ very close to the Hopf bifurcation, and in blue for $\mu\approx -0.0042$ further away from it. The code automatically adds modes when needed; the final (blue) solution has been validated with $K = 53$.
The continuous branch connecting the two solutions has been validated. The horizontal axis represents time in the original system.}
\label{fig_lor84}
% 300 iterations, ~1hour
\end{center}
\end{figure}

\begin{figure}
\begin{center}
\includegraphics[width = \textwidth]{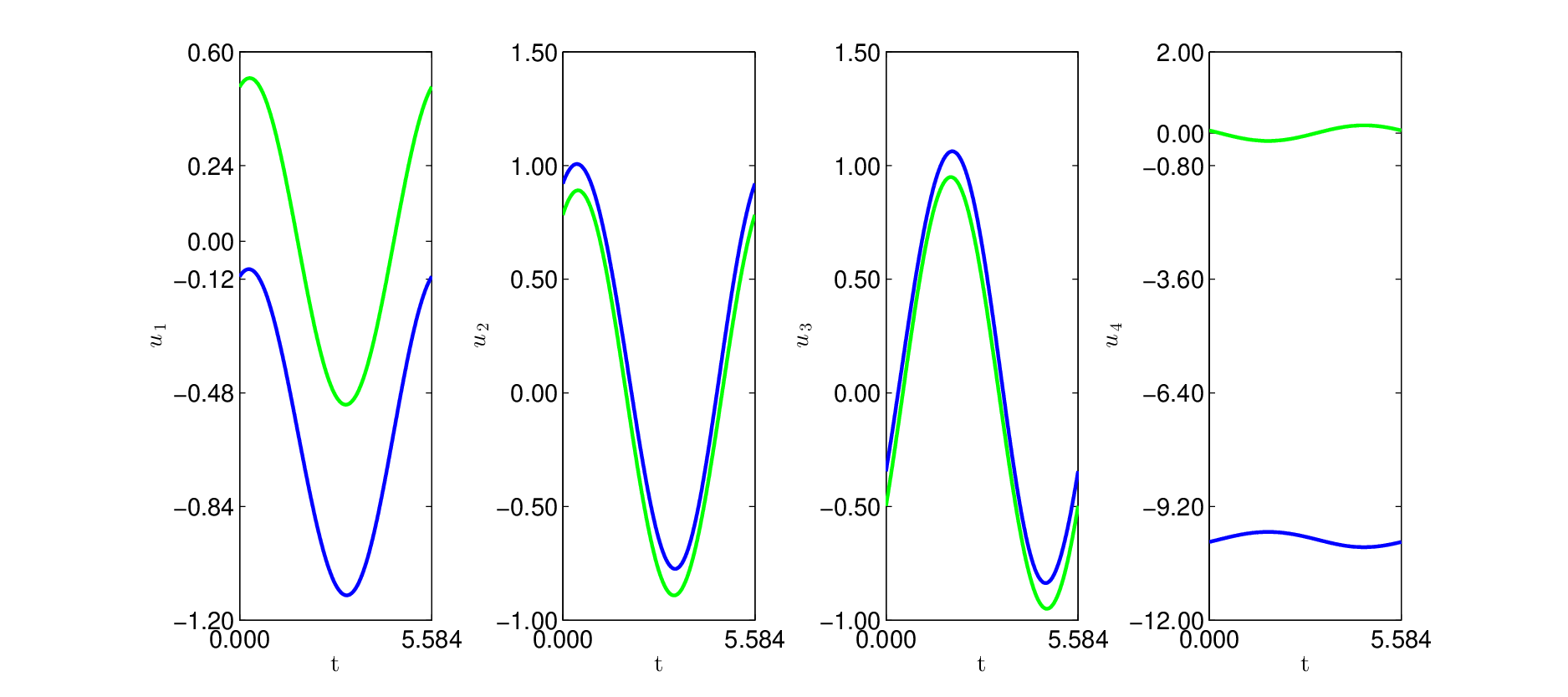}% new
\caption{The four components of the desingularized periodic solution $\bar u$ corresponding to the extended Lorenz-84 system~\eqref{lorenz84}, in green for 
$\mu\approx 0.0109$ 
%$\mu=0.010900160169807$ 
close to the Hopf bifurcation, and in blue for 
$\mu\approx -0.0108$ 
%$\mu=-0.010703201244938$ 
further away from it. 
The continuous branch connecting the two solutions has been validated. The horizontal axis represents time in the original system.}
\label{fig_lor84_2}
\end{center}
\end{figure}

\begin{figure}
\begin{center}
\includegraphics[width = 0.45\textwidth]{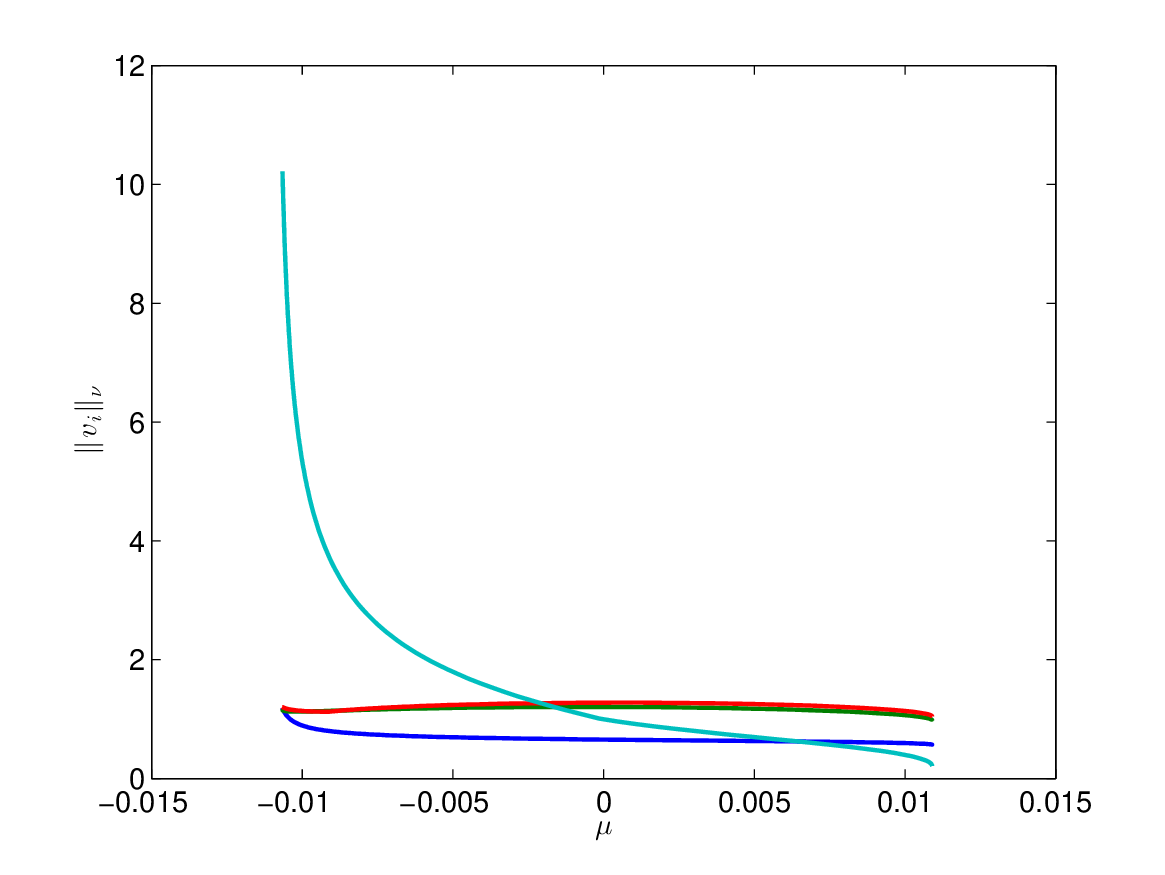}
\includegraphics[width = 0.45\textwidth]{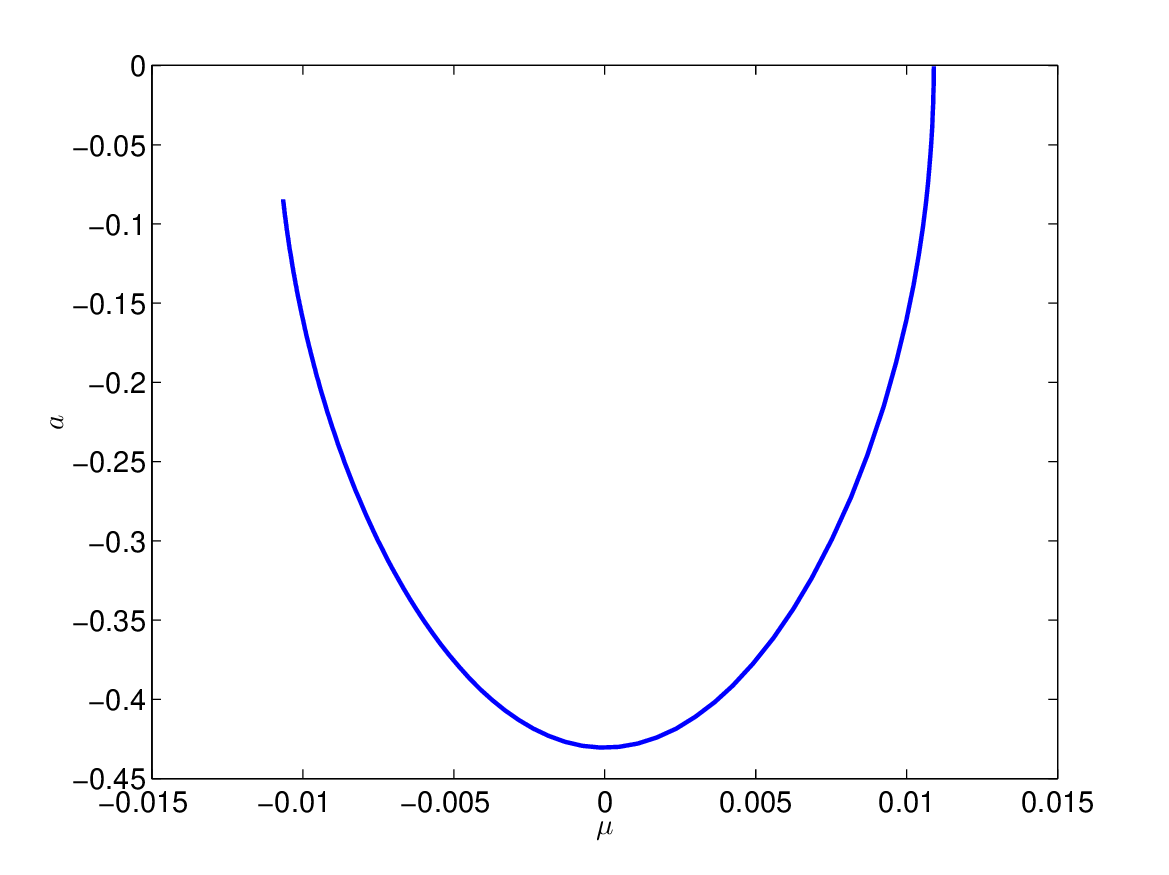}
\caption{
In the left graph the norms $\| v_i \|_\nu$, $i=1,2,3,4$ of the components of the solution branch presented in Figure \ref{fig_lor84_2} are plotted with respect to the parameter $\mu$.
On the right, the amplitude $a$ of the desingularized orbit is plotted with respect to the parameter $\mu$. }
\label{fig_lor84_2_norm}
\end{center}
\end{figure}

% \corrc Elena: why not do both then? JB <<>>
% \corrc I couldn't easily find numerically the other one, once I found one I stuck to it. I can look for the other one, if you want. Elena <<>>
% \corrc Thinking about it, I indeed think it would be good to do the other one as well. JB <<>>
% \corrc Elena: The information below is too vague. I propose to report values of $K$ and $\nu$ and intervals for $\mu$ in which we prove the Hopf bifurcation to occur, as well as the approximate equilibrium and the approximate frequency at the bifurcation. JB <<>>
% \blue{it is good enough?}
Applying the approach presented in Section~\ref{s:hopf}, we proved the existence of a Hopf bifurcation at
%\corrc Same comment about accuracy as before. JB <<>>
$$
\mu^* \in 0.05684121 + [-9.1,9.1] \cdot 10^{-6},
$$
by using the computational parameters $K = 5$ and $\nu = 1.1$.
The eigenvalue crossing the imaginary axis is $-0.5300219\:\!\imag + \imag[-4.6,4.6]\cdot 10^{-5}$. The (normalized) period of the solution at the Hopf bifurcation is 
$\tau = -1.886714 + [-1.6,1.6] \cdot 10^{-4}$. The periodic orbits bifurcates from the equilibrium $[1.197556,
  -0.033525,   0.203229,  -0.400337]+ [-1.6,1.6] \cdot 10^{-4}$.
  
Additionally, in the framework of the desingularized Hopf system 
%\corrc Was it the desingularized system or was it also extended to include first and second derivative of the curve? [I don't know why you would do that, really, but it was just not clear to me from the text]. JB <<>>
we continued the periodic orbit and the fixed point solution up to 
$\mu \approx -0.0023$,
where the periodic solution has an amplitude $a \approx 0.9394$.
In Figure~\ref{fig_lor84}, we plotted the desingularized (``blown up'') periodic solution $\bar u$ near the Hopf bifurcation (where it is unimodal) and at the end of the continuation. 

%\corrc that is a very complicated looking orbit for just $K=5$ modes (the blue one). Is it really correct? JB <<>>

The stepsize, that is, the distance between consecutive numerical approximate solutions $\|\hat x_0-\hat x_1\|_X$ used in Figure~\ref{fig_lor84} was relatively large: $h =10^{-3}$. The (refined) step where the Hopf bifurcation is proven to take place goes from amplitude $a = 4.6671\cdot 10^{-6}+[-8.1,8.1]\cdot10^{-8}$ to $a = -5.62556\cdot 10^{-5}+[-8.1,8.1]\cdot10^{-8}$.
%
%\corrc No error bounds for $a$? JB <<>>
%
We can increase the accuracy of locating the Hopf bifurcation point by adopting a smaller stepsize and/or increasing the number of modes used. With $h=10^{-3}$ and $K = 10$, we find
$$
\mu^* \in 0.05684121 + [-6.3,6.3] \cdot 10^{-6},
$$
which reflects a modest improvement.
However, with $h=10^{-5}$ and $K = 10$, we retrieve a much higher accuracy result:
$$
\mu^* \in 0.056841207164 + [-4.4,4.4] \cdot 10^{-10}.
$$
We conclude that decreasing the stepsize is the crucial factor in improving accuracy, whereas the number of modes used is less important, as may be expected since the solution is unimodal at $a=0$.

%\corrc It would be good to check that the accuracy improves for, say, $K=10$. JB <<>>
%\corrc Would also be good to report equilibrium values and frequency/period/eigenvalue at the bifurcation. JB <<>>
%
%Applying the concepts presented previously, we follow the periodic orbit generated by the Hopf bifurcation close to $\mu\approx 0.05$ and validate the Hopf bifurcation. We can then prove that the Hopf bifurcation takes place at
%$$
%\mu=0.0569. 
%$$
%%with validation bound
%%$$
%%\hat r = .
%%$$

The second Hopf bifurcation for the same values of the parameters~\eqref{e:parametervalues} takes place at 
\[
\mu \in 0.010900160+[-3.1, 3.1]\cdot 10^{-7}.
\]
The stationary solution is at $[1.079797955,
  -0.017016937,
    0.230267229,
  -0.423161664]+[-3.1, 3.1]\cdot 10^{-7}$, while the eigenvalue crossing the imaginary axes is $1.1251599\:\!\imag + \imag [-1.6,1.6]\cdot 10^{-5}$. The period of the periodic perturbation near the Hopf bifucation is $\tau = 0.888763098 +[-3.1, 3.1]\cdot 10^{-7}$.
In Figure \ref{fig_lor84_2}, we depict the desingularized periodic solution $\bar u$ near the Hopf bifurcation (where it is unimodal) and after $800$ continuation steps, where 
\[
\mu \in -0.0108027 +[-6.3,6.3]\cdot 10^{-5}.
\]
In Figure \ref{fig_lor84_2_norm}, the norm of the periodic orbit is plotted along the branch with respect to the parameter~$\mu$.
We conclude from the fact that $a$ approaches zero at both ends of the branch,
while the norms $\| v_i\|_\nu$ of the rescaled time-dependent part $\bar{u}$ explode at one end point,
 that the continuous branch of periodic solutions that originates from a Hopf point at $\mu \approx 0.0109$ terminates for $\mu \approx -0.0108$ at another Hopf point on a different branch of equilibria, not connected by continuation.

% \corrc I think we need a bit more detail about continuation of the branches of equilibria, perhaps even an additional figure? JB \blue{what do you want to add about equilibria? E} <<>>

In Figure \ref{f:lorHopf2Hopf} we depict a full continuous branch of periodic orbits of \eqref{lorenz84}, connecting these two Hopf bifurcation points.
The periodic orbits ``cross over'' from one branch of equilibria to another. Hence we use the gluing approach discussed above twice, first to switch from the desingularized system around one equilibrium to continuation of the original system, and once more to switch from the original system to the desingularized system around the other equilibrium. 

%\corrc 
%Elena: need to push the computation a little further (but I cannot run the code). 
%We should also see the other components "blow up" in the picture on the right.
%It would also be good to have the equilibria at the "end" points where $a=0$ (you of course can not quite reach $a=0$ in the blowup-endpoint now, but you probably can get quite close.) JB <<>>

%\corrc I run the computation longer (twice as long), but the last steps are so small little changed. E <<>>
%\corrc OK, just leave it at that for now. JB <<>>

% We can notice how, at the Hopf bifurcation the amplitude vanishes and the norm of the components is roughly one, as expected. Continuing along the orbit, the amplitude first increases, then tends to zero for $\mu$ tending to $-0.01$, while in the same time the norm of the components explodes. This hints to the fact that we are following a heteroclinic orbit tending to a Hopf bifurcation. The amplitude $a$ of the periodic orbit decreases, but the periodic solution $y+az$ converges to a different stationary solution $y'$, thus the norm of $z$ tends to infinity. We have not proven that the orbit connects to another Hopf bifurcation, but we consider the results here presented a good hint towards that conclusion.

The corresponding \textsc{matlab} code is provided in 
\verb+lorenz84_validated_cont.m+.
%\corrc Elena: do you mean $\bar{u}$ by "blown up" ? JB <<>>

%\corrc Elena: In Figure~\ref{fig_lor84_2_norm} on the righ, it would be more intuitive (for the reader) to depict the $a>0$ branch. JB <<>>
%\corrc that is not hte result I have at the moment E<<>>
%\corrc Right, we talked about this, just leave $a<0$ for now. JB<<>>

\subsection{A Hopf bifurcation in a hyperchaotic system}

In \cite{hyper}, the 4-dimensional (hyperchaotic) ODE system 
\begin{equation}\label{e:hyper}
\begin{cases}
\dot u_1 = a (u_1 - u_2) + u_2u_3 + u_4\\
\dot u_2 = - b u_2 + u_1u_3\\
\dot u_3 = - c u_3 + d u_1 + u_1 u_2\\
\dot u_4 = - e (u_1 +u_2)
\end{cases}
\end{equation}
is presented and studied.
The system has many interesting dynamic features, including Hopf bifurcations.
 % A hopf bifurcation is proven to occur at the origin for a given set of parameters.
%For the set of parameter
%$$
%a=0,\quad d\neq 0,\quad e>0, \quad b = c\neq e,
%$$
%a Hopf bifurcation at the origin is proven analytically. 
%\corrc Elena: that is a vague statement (what is the bifurcation parameter?) but also seems irrelevant. JB <<>>
Inspired by the analysis in~\cite{hyper} we fix
$$
b=c =1,\qquad
d=10,\qquad
e = 2,
$$
and use $\mu =  a$ as the bifurcation parameter.
% In this paper, we are interested in studying another Hopf bifurcation that happens far away from the origin for the same parameter choices. In particular, we set
% \corrc Elena: How did you get these values? From~\cite{hyper}?  JB <<>>
% \blue{I honestly set them at random considering the restrictions in the paper \cite{hyper}}
%
%
% \corrc Elena: The information below is too vague. I propose to report values of $K$ and $\nu$ and intervals for $\mu$ in which we prove the Hopf bifurcation to occur, The approximate equilibrium and the approximate frequency at the bifurcation are there, but the error bound $\hat{r}$ is missing. JB <<>>
% \blue{is it enough?}
Then, a Hopf bifurcation occurs at
%\corrc Accuracy of numbers does not make sense. JB <<>>
\[
  \mustar \in -1.01551372619 + [-2.5,2.5] \cdot 10^{-9},
\] 
from the equilibrium 
$$(u_1,u_2,u_3,u_4) \in ( 10.09901951359,
 -10.09901951359,
  -1.00000000000,
  30.61040538783) +  [-2.5,2.5] \cdot 10^{-9}, $$ where the interval notation has been slightly abused.
The normalized period at the Hopf bifurcation is 
$
  \sol{\tau}(\sstar) \in 0.68299909941 + [-2.5,2.5] \cdot 10^{-9}.
$
We continued the solution to $\mu \approx -1.043$, where the amplitude $a \approx 2.149$.
%\corrc Elena: value needed. JB <<>>
In Figure \ref{fig_hyper} we have depicted the desingularized periodic profile $\bar{u}$.
%\corrc Assuming that is what you plot. JB <<>>
%corresponding to~\eqref{e:hyper}. 
For this validation we used $K=15$ and $\nu = 1.1$. 
The corresponding \textsc{matlab} code is provided in 
\verb+main_hyper.m+.

\begin{figure}
\begin{center}
\includegraphics[width = \textwidth]{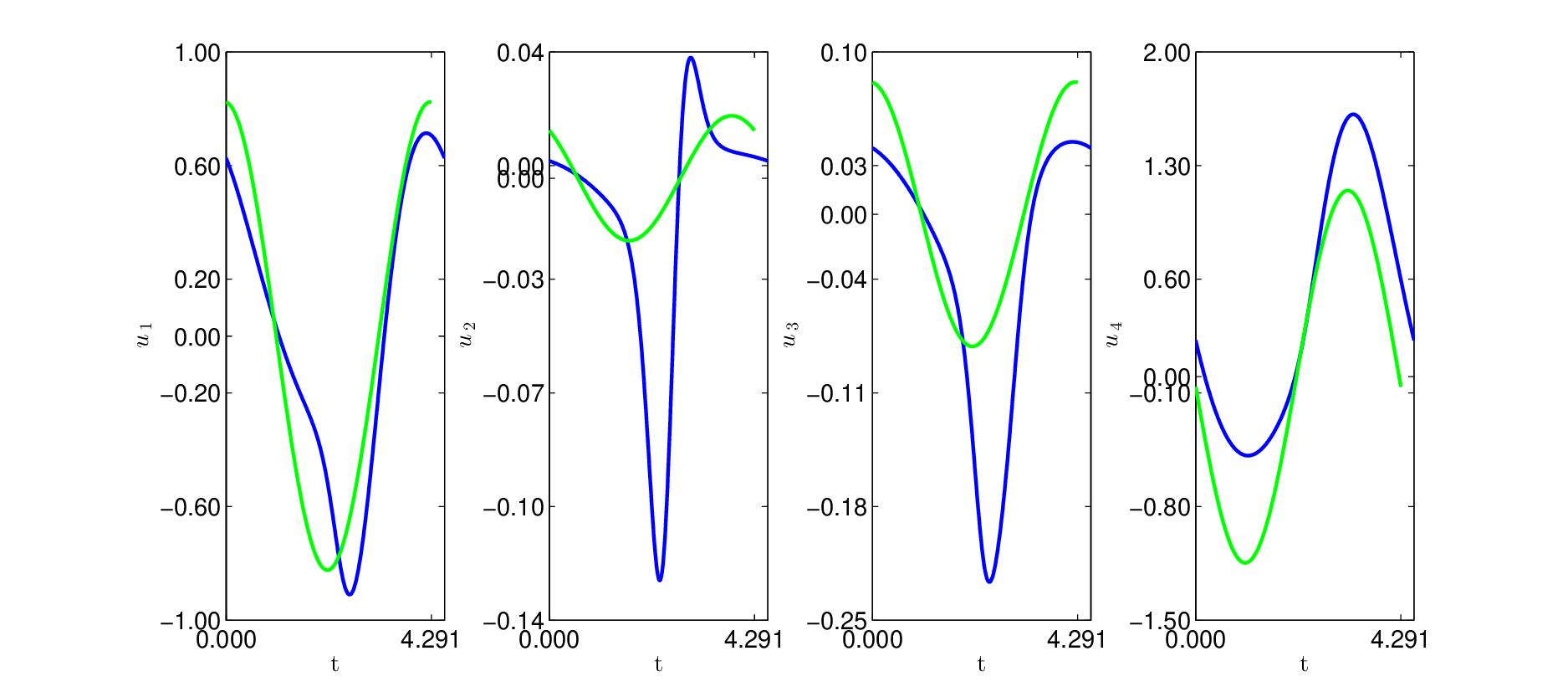}
\caption{
The four components of the desingularized periodic solution $\bar u$ corresponding to the 4-dimensional hyperchaotic system~\eqref{e:hyper}, in green for $\mu=-1.0155$ close to the Hopf bifurcation, and in blue for $\mu\approx -1.7035$ further away from it. 
The continuous branch connecting the two solutions has been validated. The horizontal axis represents time in the original system.}
\label{fig_hyper}
\end{center}
\end{figure}

\subsection{A Hamiltonian example}
\label{s:hamiltonianexample}

Consider the fourth order parabolic partial differential equation (PDE)
\begin{equation}\label{e:fourthPDE}
	u_t = -u_{xxxx} + a u_{xx} + b u + c u^2 + d u^3 .
\end{equation}
This family includes the extended Fisher-Kolmogorov and Swift-Hohenberg equations (see~\cite{PeletierTroy} and the references therein).
When studying stationary states for the problem with periodic or Neumann boundary conditions, this reduces to studying periodic solutions of the ODE 
\begin{equation}\label{e:fourthorder}
  u_{xxxx} =  a u_{xx} + b u + c u^2 + d u^3 .
\end{equation}
Even though it clashes with the notation elsewhere in the paper,
we denote by $x$ the independent variable in the ODE~\eqref{e:fourthorder},
because it fits with the PDE~\eqref{e:fourthPDE}. 
This problem is Hamiltonian with conserved quantity 
\begin{equation}\label{e:Energy}
	E =- u_{xxx}u_{x}+\frac{1}{2}(u_{xx})^2 + \frac{a}{2} (u_{x})^2      
	+\frac{b}{2}u^2+\frac{c}{3}u^3+ \frac{d}{4}u^4.
\end{equation}	
For any $b>0$ periodic orbits, which appear in $1$-parameter families due to the Hamiltonian structure, 
bifurcate from the equilibrium $u=0$ with period 
 $2\pi((a^2/4+b)^{1/2}- a/2)^{-1/2}$.
For the PDE (with periodic or Neumann boundary conditions) this means that nontrivial stationary solutions bifurcate from the trivial state when the domain size is varied.
Although this analysis can be done by hand, we use this example to illustrate how Hamiltonian problems can be brought into the framework of the current paper.

We will fix the parameters $a,b,c,d$ and introduce an \emph{artificial} continuation parameter $\mu$ to turn~\eqref{e:fourthorder} into the first order system
\begin{equation}\label{e:system4}
\begin{cases}
\dot u_1 = u_2\\
\dot u_2 = u_3\\
\dot u_3 = u_4\\
\dot u_4 = a u_3 +  b u_1 + c u_1^2 + d u_1^3 + \mu {u_2} .
\end{cases}
\end{equation}
Irrespective of $\mu$, the equilibria of the system are $u_2=u_3=u_4=0$
and $u_1$ a zero of the polynomial $p(u)=b u + c u^2 + d u^3$,  which also correspond to stationary solutions of~\eqref{e:fourthorder}.
Furthermore, we know \emph{a~priori} that $\mu=0$ for any periodic solution of~\eqref{e:system4}, since (cf.~\eqref{e:Energy})
\[
  0 = \int_0^L \frac{d}{dx} \left(-u_4 u_2 +\frac{1}{2}u_3^2 + \frac{a}{2} u_2^2      
	+\frac{b}{2}u_1^2+\frac{c}{3}u_1^3+ \frac{d}{4}u_1^4\right) dx = - \mu \int_0^L u_2^2 \, dx,   
\]
where $L$ is the period of the solution.
Hence, any periodic orbit of~\eqref{e:system4} corresponds to a periodic solution of~\eqref{e:fourthorder} (and vice versa).
The advantage of studying~\eqref{e:system4} rather than~\eqref{e:fourthorder} is that in the former the Hamiltonian structure/symmetry has been broken, and we can study it using the general continuation and bifurcation techniques from this paper. 

\begin{figure}
\begin{center}
\includegraphics[width = 0.5\textwidth]{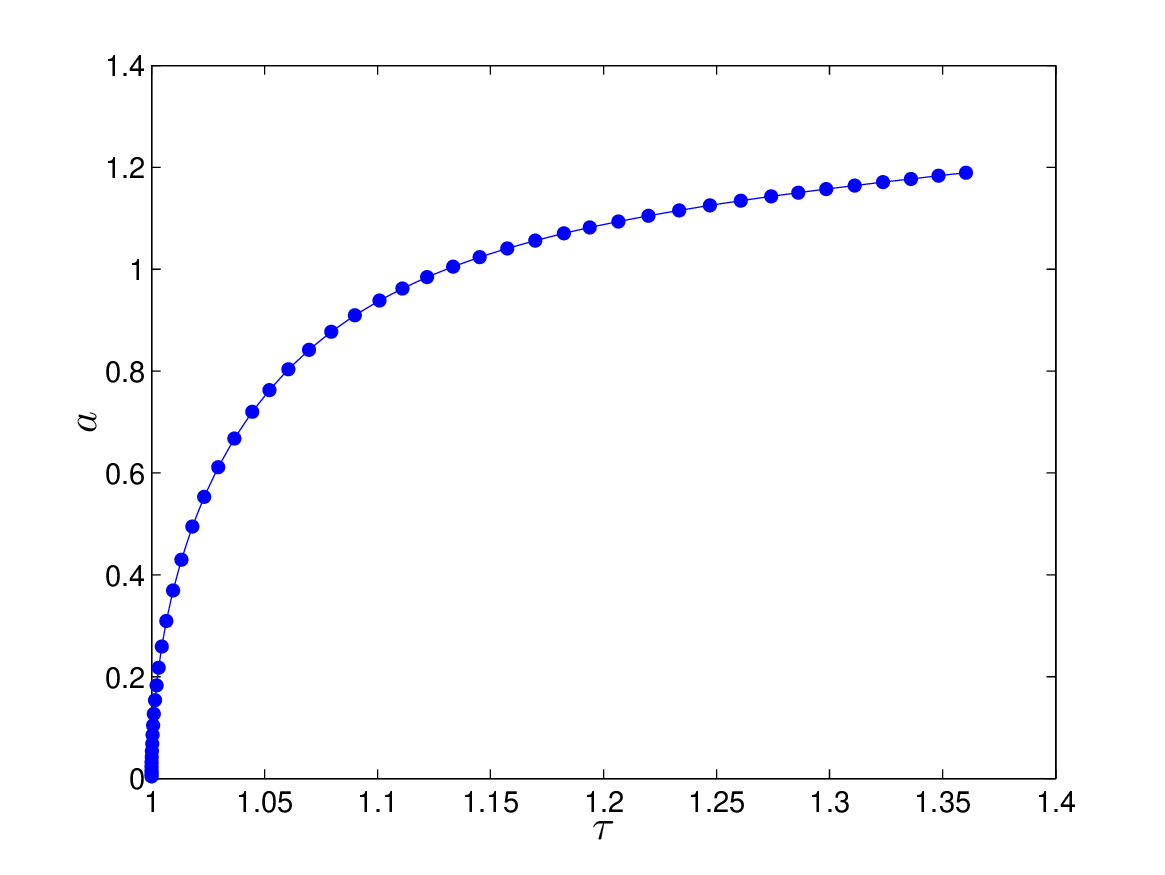}
\caption{The Hopf bifurcation, or fold for the desingularized system, with respect to the normalized period $\tau$, for~\eqref{e:system4} or, equivalently,~\eqref{e:fourthorder}. %\red{Shall we put $\tau$ as the $x$-label? JB}% The amplitude has been rescaled.
% maximum lambda_1 1.000001555814463 
%\red{Elena: the rescaling of $a$ is not necessary; just have smaller scale along the vertical axis. Also, the horizontal axis doesn't look too good. Also need to explicitly state that this is the amplitude $\sol{a}(s)$ and not the parameter $a$ in the system (apologies for the clashing notation).  JB <<>>}
}\label{f:hamilton}
\end{center}
\end{figure}

\begin{figure}
\begin{center}
\includegraphics[width = \textwidth]{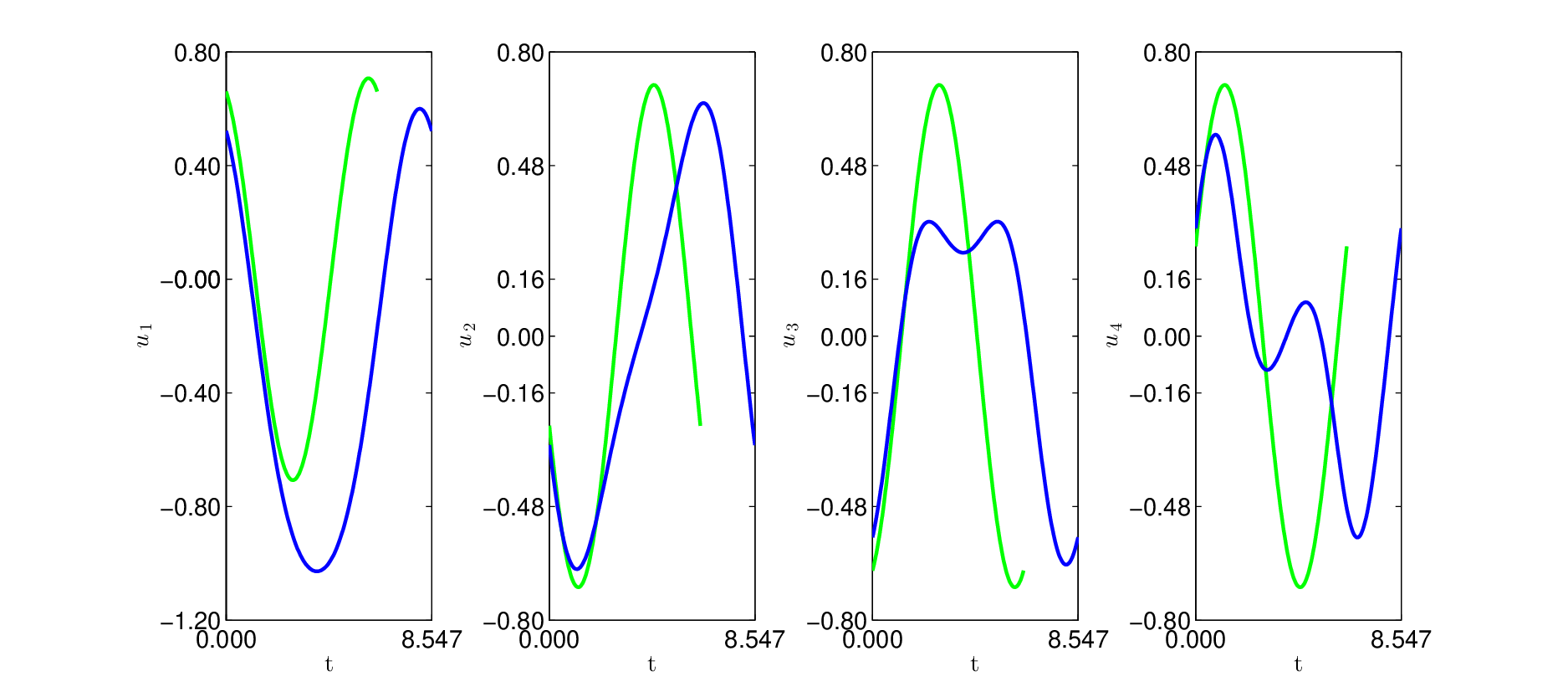}
\caption{The four components of the desingularized periodic solution $\bar u$ corresponding to~\eqref{e:system4}, in green for $\tau= 1$ very close to the Hopf bifurcation, and in blue for $\tau\approx 1.224$ further away from it. 
The continuous branch connecting the two solutions has been validated. The horizontal axis represents time in the original system.
%\red{The nonnormalized period $L=2\pi \tau$; the scaling along of the $x$-axis needs to be repaired. JB }
}
%
% The four components of solution to the desingularized Hopf problem corresponding to \eqref{e:system4} for {$\tau = 1\pm4.1\cdot 10^{-5}$} in green and {$\tau = 1.096285 \pm 4.4\cdot 10^{-5}$} in blue. The time axes hasn't been rescaled, therefore the blue period stops sooner than the green one.}
\label{f:Hamiltonian_sol}
\end{center}
\end{figure}

%\red{Elena: describe result as before}
%Just choose some $a,b,c,d$, let's do $a=2$, $b=3$, $c=1$, $d=-1$.
%\blue{Just as background: $d<0$ makes sense for the PDE; $a=2$, $b=3$ leads to the easy $\tau=1$ at the Hopf bifurcation if I did the algebra correctly; $c=1$ breaks $u\to-u$ symmetry (not so essential, but let's do it nevertheless)}.
%Find Hopf bifurcation (found at $\mu=0$ of course) and do a continuation from there.

In Figure \ref{f:hamilton}, the validated bifurcation is shown for the parameters $a=2$, $b=3$, $c=1$, $d=-1$.
The location of the bifurcation point $\sol{\tau}(\sstar)$ is validated to lie in $1 + [0.4,0.4] \cdot 10^{-4}$
using computational parameters 
$K=5$ and 
$\nu =1.1$
%\corrc Elena:  Values of $K$  and $\nu$ (which I guess is 1?) JB <<>>
(as discussed, it is analytically determined to occur at
$\lambda_1=\tau=1$).
Desingularized profiles are plotted in Figure \ref{f:Hamiltonian_sol}, continued from $\tau=1$ to $\tau \approx 1.224$. 
The corresponding \textsc{matlab} code is provided in 
\verb+main_Hamiltonian.m+.

%\corrc Elena: what "amplitude" is along the $y$-axis in Figure~\ref{f:hamilton}? If it is $a$, why does it not start at $0$? Why does the range of $\tau$ in Figure~\ref{f:hamilton} no correspond to the end point $\tau \approx 1.224$ in Figure~\ref{f:Hamiltonian_sol}. JB <<>>
%\corrc Elena: That plot should then also contain, like the other ones, a solution further away from the bifurcation; otherwise it's just sines and cosines. JB <<>>

%\corrc Elena: reference to code. JB <<>>

%\input{olderhopf}

%\input{biblio}
%\bibliographystyle{plain}
\bibliographystyle{abbrv}
\bibliography{livres}

\end{document}